\documentclass{ws-fml}
\usepackage{amsfonts}
\usepackage{amsmath}
\usepackage{amssymb}
\usepackage{latexsym}
\usepackage{graphicx}
\usepackage{bm}
\usepackage{nicefrac}
\usepackage{epstopdf}
\usepackage[usenames]{color}
\usepackage{multicol}

\input amssym.def
\newsymbol\rtimes 226F
\newfont{\nset}{msbm10}
\begin{document}
%\catchline{1}{1}{2021}{}{}

\markboth{X. Zhou \& Z. Zhang}
{Edge domination number and  the number of minimum edge dominating sets}

\title{EDGE DOMINATION NUMBER AND THE NUMBER OF MINIMUM EDGE DOMINATING SETS IN PSEUDOFRACTAL SCALE-FREE WEB AND SIERPI\'NSKI GASKET}

%%ÐÞ¸Ä

\author{XIAOTIAN ZHOU}

\address{{Shanghai Key Laboratory of Intelligent Information
Processing, School of Computer Science, Fudan University, Shanghai 200433, China}\\
\email{20210240043@fudan.edu.cn}}

\author{ZHONGZHI ZHANG\footnotemark}

\address{{Shanghai Key Laboratory of Intelligent Information
Processing, School of Computer Science, Fudan University, Shanghai 200433, China}\\
\email{zhangzz@fudan.edu.cn}}
%\affiliation{Shanghai Key Laboratory of Intelligent Information
%Processing, School of Computer Science, Fudan University, Shanghai 200433, China}

%\shortauthors{X. Zhou and Z. Zhang}
%\shorttitle{Edge domination number and  the number of minimum edge dominating sets...}%\shorttitle{Insert short title here for recto running head} %%%for recto running head
%\shortauthorlist{Xiaotian Zhou, Wanyue Xu,  Zhongzhi Zhang, and Haibin Kan} %\shortauthorlist{Insert short author list for verso running head} %%% for verso running head
%\titlerunning{short title}
%\author{{%%%% First author details
%\sc Xiaotian Zhou, Wanyue Xu,  Zhongzhi Zhang, and Haibin Kan}\\[2pt]
%Shanghai Key Laboratory of Intelligent Information Processing, School of Computer Science, Fudan University, Shanghai 200433, China\\Shanghai Engineering Research Institute of Blockchain, Fudan University, Shanghai 200433, China\\
%{\email{ zhangzz@fudan.edu.cn}}}
%{\email{16300180049@fudan.edu.cn; xuwy@fudan.edu.cn; zhangzz@fudan.edu.cn; hbkan@fudan.edu.cn}}\\[2pt]

\maketitle

\renewcommand{\thefootnote}{*}
\footnotetext[1]{Corresponding author.}
%
%\begin{history}
%\received{Day Month Year}
%\revised{Day Month Year}
%\end{history}

\begin{abstract}
As a fundamental research object, the minimum edge dominating set (MEDS) problem is of both theoretical and practical interest. However,  determining  the size of a MEDS and the number of all MEDSs in a general graph  is NP-hard, and it thus makes sense to find special graphs for which the MEDS problem can be exactly solved. In this paper, we study analytically  the MEDS problem in the pseudofractal scale-free web and the Sierpi\'nski gasket with the same  number of vertices and  edges. For both graphs, we obtain  exact expressions for  the edge domination number, as well as recursive solutions to the number of distinct MEDSs. In the pseudofractal scale-free web, the edge  domination number is one-ninth of the number of edges, which is three-fifths of the  edge  domination number of the Sierpi\'nski gasket. Moreover, the number of  all MEDSs in the pseudofractal scale-free web is also less than that corresponding to the  Sierpi\'nski gasket. We argue that the difference of the size and number of MEDSs between the two studied graphs lies in the scale-free topology.
%%%% If classification number provided then
\end{abstract}

\newpage

\begin{multicols}{2}

\section{INTRODUCTION}

An edge dominating set (EDS) of a graph $\mathcal{G}$ with edge set $\mathcal{E}$ is a subset $\mathcal{F}$ of $\mathcal{E}$, such that each edge in $\mathcal{E}\setminus \mathcal{F}$ is adjacent to at least one edge belonging to $\mathcal{F}$~\cite{Mi77}. An independent edge dominating set is an EDS in which no two edges are adjacent. We call $\mathcal{F}$ a minimum edge dominating set (MEDS) if it has the smallest cardinality.  The cardinality of a MEDS is called the edge domination number of graph $\mathcal{G}$, which equals the size of  any minimum independent EDS of $\mathcal{G}$.   To find a MEDS of a given graph, known as  the MEDS problem, is a basic  graph problem and plays a critical role in graph algorithms.  The MEDS problem arises in various practical scenarios. For example, monitoring the state of communications  taking place among vertices of a wireless ad hoc network can be formulated as the MEDS problem~\cite{PaEsKaTa19}.  Moreover, The MEDS problem is closely related to many other graph problems, like the vertex cover problem and independent set problem, the latter of which has found  applications in different fields, e.g., coding theory~\cite{BuPaSeShSt02}, collusion detection in voting pools~\cite{ArFaDoSiKo11}, and scheduling in wireless networks~\cite{JoLiRySh16}.

Given its intrinsic importance in theoretical and practical scenarios, the MEDS problem has attracted considerable attention from various disciplines~\cite{CaLaLe09,ScVi12,XiNa13,XiNa14,EsMoPaXi15,GoHeKrVi15,FuSh18,ItKaKaKoOk19}. %, e.g., theoretical computer science.
However, the MEDS problem is one of the basic NP-hard problems. It is well-known that the problem is NP-hard even when the graph is limited to planar or bipartite graphs of maximum degree three~\cite{YaGa80}. In particular, counting all MEDSs in a graph is even more difficult, which is often \#P-complete~\cite{Va79TCS,Va79SiamJComput}. Thanks to the hardness of the MEDS problem in a generic graph, many efforts have been devoted to the problem in special graph classes, e.g., planar graph~\cite{WaYaGuCh13}, where the problem is also called matrix domination. Of particular interest for NP-hard and \#P-complete problems is to design or find specific graphs, where the problems can be solved exactly~\cite{LoPl86}. However, to the best of our knowledge,  there is  still a lack of rigorous results about edge domination number and the number of MEDSs.

On the other hand, extensive empirical work~\cite{Ne03} has demonstrated that a broad range of real-world networks display the typical scale-free property~\cite{BaAl99}, that is,  the degrees of their vertices follow a power-law distribution $P(k) \sim k^{-\gamma}$. It has been shown that this nontrivial feature has a great influence on diverse  structural, combinatorial, and dynamical properties of a power-law graph, such as average distances~\cite{ChLu02}, maximum matchings~\cite{LiSlBa11,ZhWu15}, dominating sets~\cite{GaHaK15,JiLiZh17}, epidemic spreading~\cite{ChWaWaLeFa08}, and noisy consensus~\cite{YiZhPa19}. Although there has been much interest in studying the MEDS problem in general graphs or some particular graphs, no existing work considered the MEDS problem for scale-free graphs. At present, the effect of scale-free topology on the MEDS problem is not understood, although it is expected to play a central role in the MEDS problem and to be significant for understanding the practical applications of MEDS problem in power-law graphs. Particularly, exact results about the edge domination number and the number of different MEDSs in a scale-free network is still lacking, in spite of the fact that exact results help to test approximation algorithms~\cite{CaLaLe09,ScVi12,FuSh18} for the MEDS problem.

In this paper, we study the edge domination number and the number of MEDSs in a scale-free graph, called pseudofractal scale-free web~\cite{DoGoMe02}. For comparison, we also study related quantities for the Sierpi\'nski gasket~\cite{FuSh92} that has found wide applications~\cite{ShWaChDaetal15,MoChDaWuWa19,Jietal20} and has received considerable attention~\cite{LaLuSu09,EtCa18,EtCa19,QiZh19,QiDoZhZh20,ChChXuZh20,WuChZhZh20}. Although both studied networks are self-similar and have the same number of vertices and edges, the pseudofractal scale-free web is heterogeneous, in sharp contrast to the Sierpi\'nski gasket that is homogeneous. By using a decimation technique~\cite{KnVa86} to both self-similar graphs, we find the exact edge domination number, as well the recursion solutions for the number of all possible MEDSs. The edge domination number of the pseudofractal scale-free web is three-fifths of that corresponding to the Sierpi\'nski gasket. In addition, the number of MEDSs of the former is also less than of the latter, although in both graphs, the number of MEDSs increases as an exponential function of the number of edges. We show that the architecture dissimilarity between the two nontrivial studied graphs is responsible for their difference in edge domination number and the number of MEDSs.

\section{EDGE DOMINATION NUMBER AND MINIMUM EDGE DOMINATING SETS IN PSEUDOFRACTAL SCALE-FREE WEB}

In this section, we determine the edge domination number and the number of minimum edge dominating sets  in the pseudofractal scale-free web.

\subsection{Network construction and properties}

The pseudofractal scale-free web~\cite{DoGoMe02}  is created using an iterative approach. Let  $\mathcal{G}_n$, $n\geq 1$, denote the $n$th generation network. Then the scale-free network is generated as follows. When $n=1$, $\mathcal{G}_1$ is a 3-clique, the complete graph of 3 vertices. For $n > 1$, $\mathcal{G}_n$ is obtained from $\mathcal{G}_{n-1}$ by performing the following operations: For each existent edge in $\mathcal{G}_{n-1}$, a new vertex is  created and linked to both endvertices of this edge. Figure~\ref{network} illustrates the networks for the first several generations. By construction, it is straightforward to  very that the number of edges in $\mathcal{G}_n$ is $E_n=3^{n+1}$.

%%%%%%%%%%%%%%%%%%%%%%%%%%%%%%%%%%%%%%%%%%%%%%%%%%%%%%%%%
% Figure  1
%%%%%%%%%%%%%%%%%%%%%%%%%%%%%%%%%%%%%%%%%%%%%%%%%%%%%%%%%%
%\begin{figure*}
%\begin{center}
%\includegraphics{LinY_1.eps}
%\end{center}
%\caption[kurzform]{The first three generations of the pseudofractal scale-free web.} \label{network}
%\end{figure*}

\begin{figurehere}
\centerline{
\includegraphics[width=.70\linewidth]{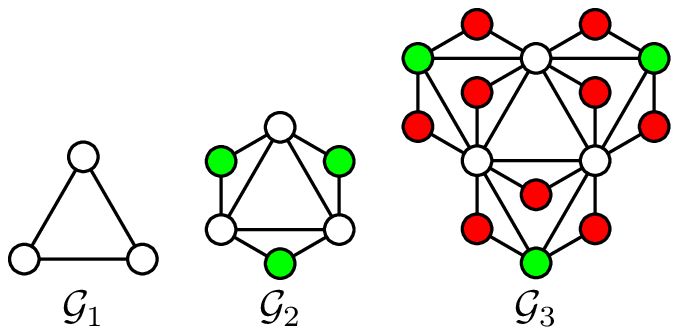}
}
\caption[kurzform]{The first three generations of the pseudofractal scale-free web.} \label{network}
\end{figurehere}
%%%%%%%%%%%%%%%%%%%%%%%%%%%%%%%%%%%%%%%%%%%%%%%%%%%%%%%%%%

%%%%%%%%%%%%%%%%%%%%%%%%%%%%%%%%%%%%%%%%%%%%%%%%%%%%%%%%%
% Figure  2
%%%%%%%%%%%%%%%%%%%%%%%%%%%%%%%%%%%%%%%%%%%%%%%%%%%%%%%%%%
%\begin{figure*}
%\begin{center}
%\includegraphics[width=.70\linewidth]%,trim=120 65 150 520]
%{SFsimilar.eps}
%\caption{Another construction of the pseudofractal scale-free web, highlighting its self-similarity.} \label{mergeF}
%\end{center}
%\end{figure*}

\begin{figurehere}
\centerline{
\includegraphics[width=.90\linewidth]%,trim=120 65 150 520]
{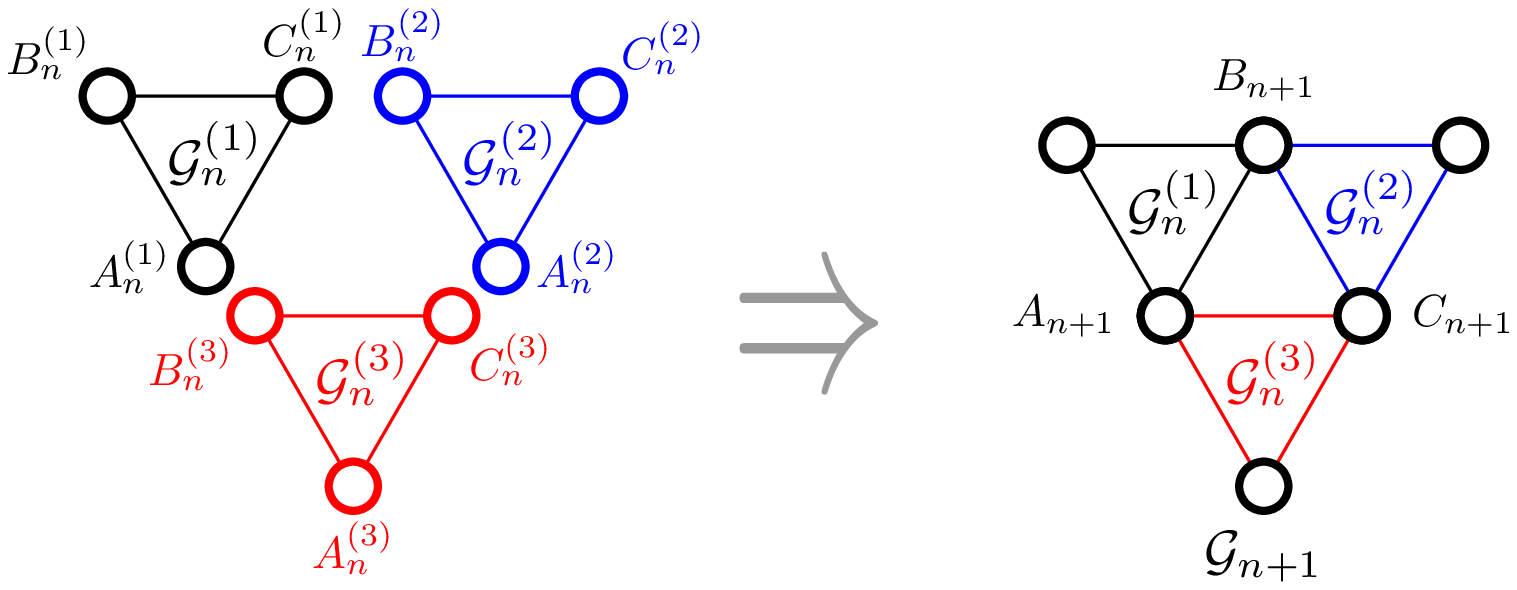}
}
\caption{Another construction of the pseudofractal scale-free web, highlighting its self-similarity.} \label{mergeF}
\end{figurehere}
%%%%%%%%%%%%%%%%%%%%%%%%%%%%%%%%%%%%%%%%%%%%%%%%%%%%%%%%%%

The pseudofractal scale-free web exhibits the prominent properties observed in a majority of real-world networks. First, it is scale-free, since  its vertex degrees  obey a distribution of power-law form $P(k)\sim k^{1+\ln 3/\ln 2}$~\cite{DoGoMe02}. Moreover, it displays the small-world effect, since its average path length grows logarithmically with the number of vertices~\cite{ZhZhCh07} and its average clustering coefficient tends to a large constant $0.8$.

Another striking property of the pseudofractal scale-free web is its self-similarity, which is also ubiquitous  in real-life systems~\cite{SoHaMa05}. In $\mathcal{G}_n$, the degree of the initial three vertices in $\mathcal{G}_1$ is the largest, we thus call them hub vertices, and denote them by $A_{n}$, $B_{n}$, and $C_{n}$, respectively. The self-similarity suggests an alternative construction  of pseudofractal scale-free web~\cite{ZhZhCh07}. Given network $\mathcal{G}_{n}$, one can obtain $\mathcal{G}_{n+1}$   by merging three replicas of $\mathcal{G}_{n}$ at their hub vertices, see Fig.~\ref{mergeF}. Denote  $\mathcal{G}_{n}^{(\theta)}$, $\theta=1,2,3$ as the three copies of $\mathcal{G}_{n}$, and denote their   hub vertices by $A_{n}^{(\theta)}$, $B_{n}^{(\theta)}$, and $C_{n}^{(\theta)}$, respectively. Then, $\mathcal{G}_{n+1}$ can be obtained by joining $\mathcal{G}_{n}^{(\theta)}$, with $A_{n}^{(1)}$ (resp. $C_{n}^{(1)}$, $A_{n}^{(2)}$) and $B_{n}^{(3)}$ (resp. $B_{n}^{(2)}$, $C_{n}^{(3)}$) being identified as the hub vertex $A_{n+1}$ (resp. $B_{n+1}$, $C_{n+1}$) in $\mathcal{G}_{n+1}$.

%%%%%%%%%%%%%%%%%%%%%%%%%%%%%%%%%%%%%%%%%%%%%%%%%%%%%%%%%%
%% Figure  2
%%%%%%%%%%%%%%%%%%%%%%%%%%%%%%%%%%%%%%%%%%%%%%%%%%%%%%%%%%%
%%\begin{figure*}
%%\begin{center}
%%\includegraphics[width=.70\linewidth]%,trim=120 65 150 520]
%%{SFsimilar.eps}
%%\caption{Another construction of the pseudofractal scale-free web, highlighting its self-similarity.} \label{mergeF}
%%\end{center}
%%\end{figure*}
%
%\begin{figurehere}
%\centerline{
%\includegraphics[width=.90\linewidth]%,trim=120 65 150 520]
%{SFsimilar.eps}
%}
%\caption{Another construction of the pseudofractal scale-free web, highlighting its self-similarity.} \label{mergeF}
%\end{figurehere}
%%%%%%%%%%%%%%%%%%%%%%%%%%%%%%%%%%%%%%%%%%%%%%%%%%%%%%%%%%%

Let $N_n$ denote the number of vertices in $\mathcal{G}_n$. According to the second construction approach of the network, $N_n$ follows the relation $N_{n+1}=3N_n-3$, which under the initial value $N_1=3$ is solved to yield $N_n=(3^{n}+3)/2$.

\subsection{Edge domination number and the number of minimum edge dominating sets}

Note that for any EDS $\chi$ of $\mathcal{G}_n$, there are  three possible states for each of three hub vertices by considering whether its incident edges are in  $\chi$ or not. For the first state, at least one incident edge belongs to  $\chi$, we denote this state by a filled hub $\bullet$. For the second state, all its incident edges do not belong to  $\chi$,  but are  adjacent to some edges in $\chi$.  We denote this state by an empty hub $\bigcirc$. While for the third state, all its incident edges do not belong to  $\chi$,  but only a part  are dominated by other edges in $\chi$.  We denote this state by a cross hub $\otimes$.

Let $\gamma_n$ denote the edge domination number of  $\mathcal{G}_n$. In order to determine $\gamma_n$, we define some intermediate quantities.  First, according to the state of hub vertices, we classify  all the EDSs of $\mathcal{G}_n$  into four classes: $\mathcal{A}_n$, $\mathcal{B}_n$, $\mathcal{C}_n$ and $\mathcal{D}_n$. For each EDS in $\mathcal{A}_n$, there is no filled hub.  For each EDS in  $\mathcal{B}_n$, there is one  and only one  filled hub. While for each EDS in $\mathcal{C}_n$ ($\mathcal{D}_n$),  there are exactly two (three) filled hub vertices. Moreover,  for $\mathcal{A}_n$, $\mathcal{B}_n$, $\mathcal{C}_n$ and $\mathcal{D}_n$, we can further define some subsets of them  with the smallest cardinality: $\mathcal{A}_n^i$, $\mathcal{B}_n^i$, $\mathcal{C}_n^i$ and $\mathcal{D}_n^i$, $i\in \{ 0, 1, 2,3\}$, where $i$ means that there is/are exact $i$ cross hub vertex/vertices in the corresponding EDS. For example, $\mathcal{B}_n^2$ denotes the subset of $\mathcal{B}_n$ such that  for each EDS in $\mathcal{B}_n^2$, it has  exactly  two cross hub vertices and the smallest cardinality compared with other EDSs in $\mathcal{B}_n$. Finally,  Let $a_n^i$,  $b_n^i$, $c_n^i$, and $d_n^i$  represent be the cardinality of $\mathcal{A}_n^i$, $\mathcal{B}_n^i$, $\mathcal{C}_n^i$, $i=0, 1, 2, 3$, and $\mathcal{D}_n^i$, respectively.  By definition, we have the following lemma.
\begin{lemma}\label{Dom01}
For $\mathcal{A}_n^i$, $\mathcal{B}_n^i$, $\mathcal{C}_n^i$ and $\mathcal{D}_n^i$,  $n\geq 1$ and $i\in \{ 0, 1, 2,3\}$, only $\mathcal{B}_n^2$, $\mathcal{C}_n^0$, $\mathcal{C}_n^1$ and $\mathcal{D}_n^0$ are existent. Thus, $\gamma_n = \min\{c_n^0,  d_n^0\}$.
\end{lemma}

%%%%%%%%%%%%%%%%%%%%%%%%%%%%%%%%%%%%%%%%%%%%%%%%%%%%%%%%%
% Figure 3
%%%%%%%%%%%%%%%%%%%%%%%%%%%%%%%%%%%%%%%%%%%%%%%%%%%%%%%%%%
%\begin{figure*}
%\begin{center}
%\includegraphics[width=0.70\linewidth]%,trim=0 80 0 500]
%{base_1.eps}
%\end{center}
%\caption[kurzform]{\label{Demo01} Illustrations for the definitions of
%$C_n^0$, $D_n^0$, $C_n^1$ and $B_n^2$, only showing the three hub vertices. (a), (b), (c) and (d) represent an EDS set belonging to $C_n^0$, $D_n^0$, $C_n^1$ and $B_n^2$, respectively.} %Every filled circle denotes a hub which has at least an edge connected to it is in the edge dominating set, each empty circle means none of its edge is in the edge dominating set but all of them are adjacent to some edges in the edge dominating set, and the cross circle means some of the edges connected to it are not adjacent to any of the edge in the edge dominating set.
%\end{figure*}

\begin{figurehere}
\centerline{
\includegraphics[width=0.90\linewidth]%,trim=0 80 0 500]
{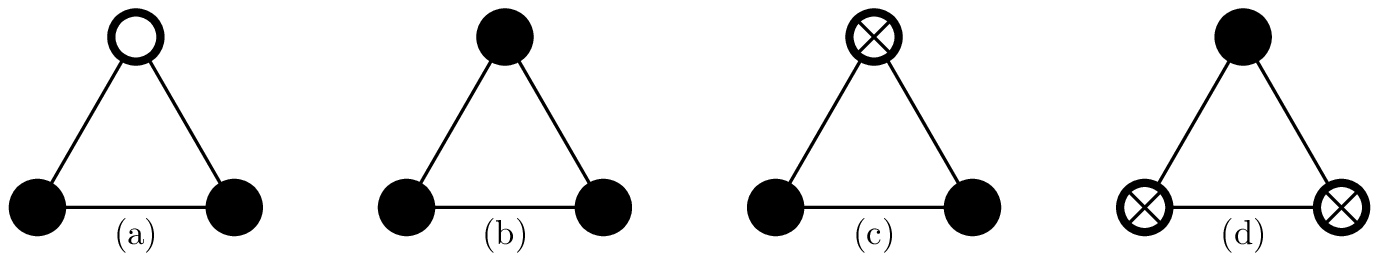}
}
\caption[kurzform]{\label{Demo01} Illustrations for the definitions of
$C_n^0$, $D_n^0$, $C_n^1$ and $B_n^2$, only showing the three hub vertices. (a), (b), (c) and (d) represent an EDS set belonging to $C_n^0$, $D_n^0$, $C_n^1$ and $B_n^2$, respectively.}
\end{figurehere}
%%%%%%%%%%%%%%%%%%%%%%%%%%%%%%%%%%%%%%%%%%%%%%%%%%%%%%%%%%

Figure~\ref{Demo01} illustrates the   definitions of $\mathcal{C}_n^0$, $\mathcal{D}_n^0$, $\mathcal{C}_n^1$ and $\mathcal{B}_n^2$.

Having reduced the problem of determining $\gamma_n$ to computing $c_n^0$ and $d_n^0 $, we now determine these two quantities. For this purpose, we alternatively evaluate the  four quantities $c_n^0$, $d_n^0 $,  $c_n^1$ and $b_n^2$,   by using the self-similar property of the pseudofractal scale-free web, since they themselves can be computed recursively.

%%%%%%%%%%%%%%%%%%%%%%%%%%%%%%%%%%%%%%%%%%%%%%%%%%%%%%%%%
% Figure  4
%%%%%%%%%%%%%%%%%%%%%%%%%%%%%%%%%%%%%%%%%%%%%%%%%%%%%%%%%%
\begin{figure*}[htbp]
\centering
\includegraphics[width=0.75\linewidth]{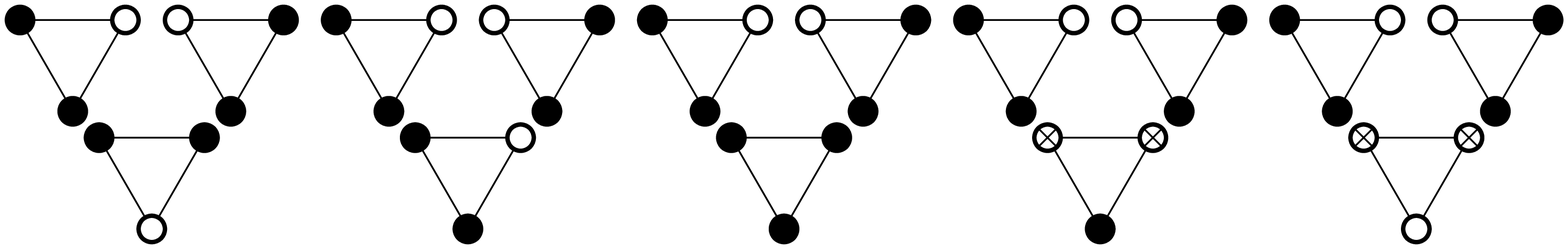}%,trim=0 80 0 500
\caption{\label{Theta0}Illustration of all possible configurations of EDSs in $\mathcal{C}_{n+1}^0$ for $\mathcal{G}_{n+1}$. Only the hub vertices of $\mathcal{G}_{n}^{(\theta)}$, $\theta=1,2,3$, are shown.  Note that here we only illustrate the  EDSs, each of which includes two filled hub vertices $A_{n+1}$ and $C_{n+1}$, and one empty hub vertex  $B_{n+1}$.  Similarly, we can illustrate those EDSs, each  including two filled hub vertices  $B_{n+1}$ and $C_{n+1}$ (resp. $A_{n+1}$ and $B_{n+1}$), and one empty hub vertex  $A_{n+1}$ (resp. $C_{n+1}$) .}
\end{figure*}
%%%%%%%%%%%%%%%%%%%%%%%%%%%%%%%%%%%%%%%%%%%%%%%%%%%%%%%%%%

\begin{lemma}\label{edgDom02}
For two successive generation networks $\mathcal{G}_n$ and $\mathcal{G}_{n+1}$, $n\geq1$,
\begin{equation}\label{Dom02}
c_{n+1}^0 = \min\{3c_n^0,2c_n^0+d_n^0,2c_n^0+b_n^2,2c_n^0+c_n^1\},
\end{equation}
\begin{align}\label{Dom03}
d_{n+1}^0 &= \min\{3c_n^0,2c_n^0+d_n^0,2c_n^0+b_n^2,2c_n^0+c_n^1,3c_n^1,\nonumber\\
&\quad c_n^0+2d_n^0,c_n^0+2c_n^1,c_n^0+d_n^0+c_n^1,2d_n^0+c_n^1,\nonumber\\
&\quad c_n^0+c_n^1+b_n^2,3d_n^0,2d_n^0+b_n^2,c_n^0+d_n^0+b_n^2,\nonumber\\
&\quad d_n^0+b_n^2+c_n^1,d_n^0+2c_n^1\},
\end{align}
\begin{align}\label{Dom04}
c_{n+1}^1 &= \min\{3c_n^1,2c_n^1+c_n^0,2c_n^1+d_n^0,c_n^1+2c_n^0,\nonumber\\
&\quad c_n^1+c_n^0+d_n^0,b_n^2+c_n^0+c_n^1,b_n^2+c_n^0+d_n^0,\nonumber\\
&\quad b_n^2+2c_n^0,2b_n^2+c_n^0,2b_n^2+d_n^0,b_n^2+2c_n^1,\nonumber\\
&\quad b_n^2+c_n^1+d_n^0\},
\end{align}
\begin{equation}\label{Dom05}
b_{n+1}^2 = \min\{2c_n^0+b_n^2,c_n^0+b_n^2+c_n^1,2b_n^2+c_n^1,b_n^2+2c_n^1\}.
\end{equation}
\end{lemma}
\begin{proof}
By definition, $b_{n+1}^2$, $c_{n+1}^0$,  $c_{n+1}^1$, and $d_{n+1}^0 $ are   the cardinality of sets $\mathcal{B}_{n+1}^2$, $\mathcal{C}_{n+1}^0$, $\mathcal{C}_{n+1}^1$, and $\mathcal{D}_{n+1}^0$, respectively. Below, we will show that each of these four sets $\mathcal{B}_{n+1}^2$, $\mathcal{C}_{n+1}^0$, $\mathcal{C}_{n+1}^1$ and $\mathcal{D}_{n+1}^0$ can be constructed iteratively from $\mathcal{B}_n^2$, $\mathcal{C}_n^0$, $\mathcal{C}_n^1$ and $\mathcal{D}_n^0$.   Then, we express $b_{n+1}^2$, $c_{n+1}^0$, $c_{n+1}^1$ and $d_{n+1}^0$ in terms of  $b_n^2$, $c_n^0$, $c_n^1$ and $d_n^0$.

We first consider Eq.~\eqref{Dom02}, which can be proved graphically.

Note that $\mathcal{G}_{n+1}$ is composed of three copies of $\mathcal{G}_n$.  Figure~\ref{Theta0} illustrates   all possible configurations of  EDSs in $\mathcal{C}_{n+1}^0$ for $\mathcal{G}_{n+1}$. From Fig~\ref{Theta0}, we obtain
\begin{equation*}
c_{n+1}^0 = \min\{3c_n^0,2c_n^0+d_n^0,2c_n^0+b_n^2,2c_n^0+c_n^1\}.
\end{equation*}

For Eqs.~\eqref{Dom03},~\eqref{Dom04}, and~\eqref{Dom05}, they can be proved similarly. In Figs.~\ref{Theta1},~\ref{Theta2}, and~\ref{Theta3}, we provide graphical representations of Eqs.~\eqref{Dom03},~\eqref{Dom04}, and~\eqref{Dom05}, respectively.
\end{proof}
%we don't know whether they are in the domination set. Therefore, We have to consider all configurations in each situation. there are still three hub vertices of these copies.

%As we mentioned above, there are four different situations for configuration of hub vertices. However, except the new hub vertices of $G_{n+1}$, there are still three hub vertices of these copies. For these three vertices, we don't know whether they are in the domination set. Therefore, We have to consider all configurations in each situation.\par

%First, all three hub vertices of $G_{n+1}$ are not in domination set. Considering symmetry and rotation, there are four intrinsic different configurations in showed in the Figure1.

%Similarly, we can enumerate all configurations of other three situations.

%%%%%%%%%%%%%%%%%%%%%%%%%%%%%%%%%%%%%%%%%%%%%%%%%%%%%%%%%
% Figure  5
%%%%%%%%%%%%%%%%%%%%%%%%%%%%%%%%%%%%%%%%%%%%%%%%%%%%%%%%%%
\begin{figure*}[htbp]
\centering
\includegraphics[width=0.85\linewidth]{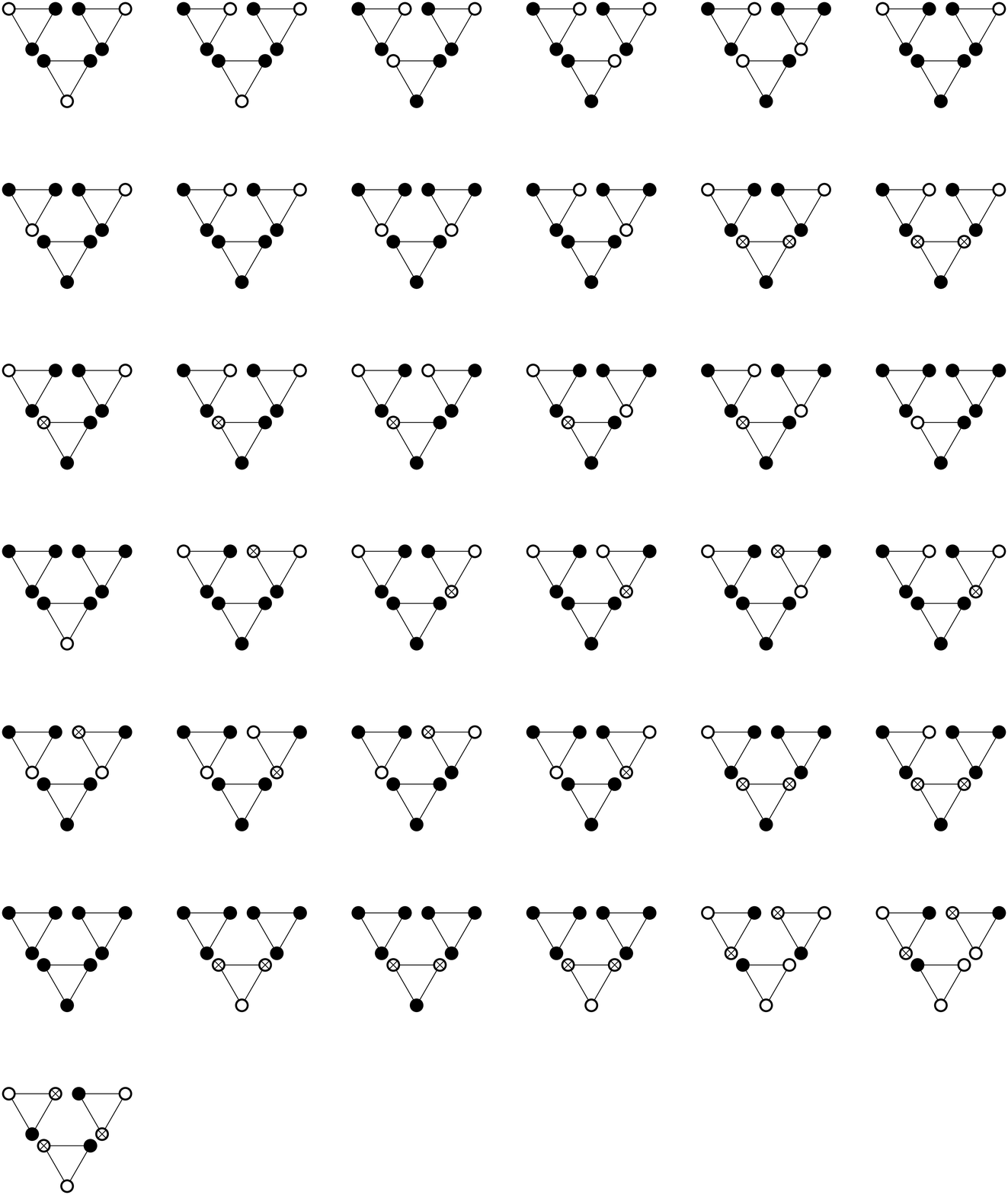}
\caption{\label{Theta1}Illustration of all possible configurations of EDSs in $\mathcal{D}_{n+1}^0$ for $\mathcal{G}_{n+1}$. Note that  we omit those configures that are  rotationally equivalent to those illustrated here.}
\end{figure*}
%%%%%%%%%%%%%%%%%%%%%%%%%%%%%%%%%%%%%%%%%%%%%%%%%%%%%%%%%%

%%%%%%%%%%%%%%%%%%%%%%%%%%%%%%%%%%%%%%%%%%%%%%%%%%%%%%%%%
% Figure  6
%%%%%%%%%%%%%%%%%%%%%%%%%%%%%%%%%%%%%%%%%%%%%%%%%%%%%%%%%%
\begin{figure*}[htbp]
\centering
\includegraphics[width=0.75\linewidth]{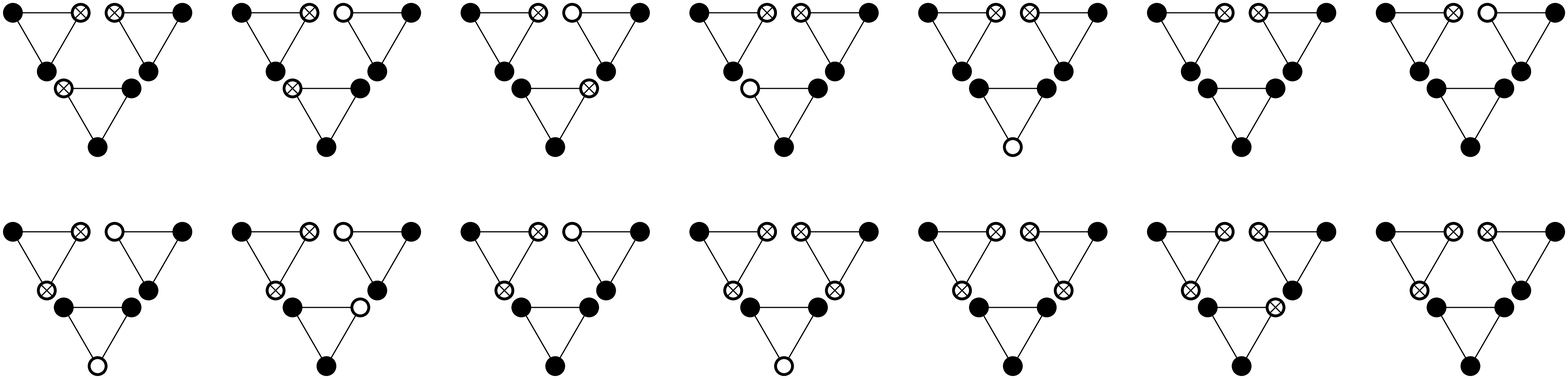}
\caption{\label{Theta2}Illustration of all possible configurations of EDSs in $\mathcal{C}_{n+1}^1$ for $\mathcal{G}_{n+1}$. Note that  we omit those configures that are  rotationally equivalent to those illustrated here.}
\end{figure*}
%%%%%%%%%%%%%%%%%%%%%%%%%%%%%%%%%%%%%%%%%%%%%%%%%%%%%%%%%%

%%%%%%%%%%%%%%%%%%%%%%%%%%%%%%%%%%%%%%%%%%%%%%%%%%%%%%%%%
% Figure  7
%%%%%%%%%%%%%%%%%%%%%%%%%%%%%%%%%%%%%%%%%%%%%%%%%%%%%%%%%%
\begin{figure*}[htbp]
\centering
\includegraphics[width=0.7\linewidth]{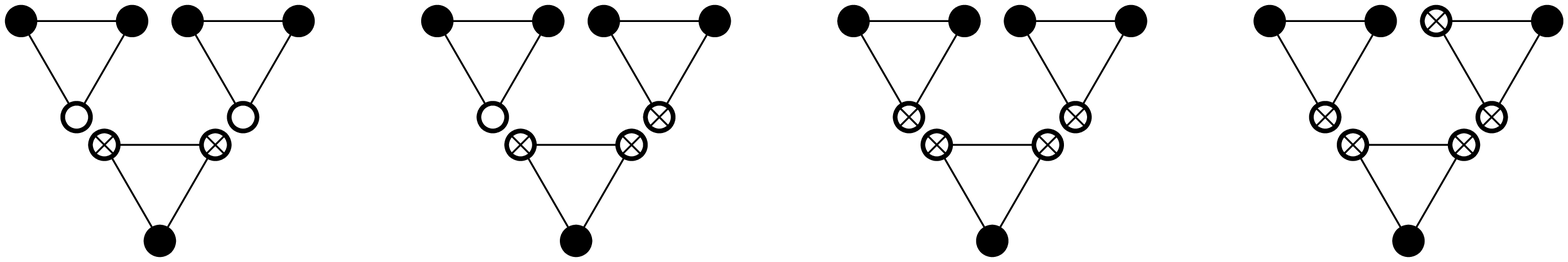}
\caption{\label{Theta3}Illustration of all possible configurations of EDSs in $\mathcal{B}_{n+1}^2$ for $\mathcal{G}_{n+1}$. Note that  we omit those configures that are  rotationally equivalent to those illustrated here.}
\end{figure*}
%%%%%%%%%%%%%%%%%%%%%%%%%%%%%%%%%%%%%%%%%%%%%%%%%%%%%%%%%%

\begin{lemma}\label{Dom06}
For network $\mathcal{G}_n$, $n\geq 3$, $c_n^0 > b_n^2 > c_n^1 = d_n^0$.
\end{lemma}
\begin{proof}
We will prove this lemma by mathematical induction on $n$. For $n=3$, we  obtain $c_3^0=5$, $b_3^2=4$, $d_3^0=3$ and $c_3^1=3$ by hand. Thus, the basis step holds immediately.

Suppose that the statement  holds for $n=t$, $t\geq 3$. Then, from Eq.~\eqref{Dom05}, $b_{t+1}^2 = \min\{2c_t^0+b_t^2,c_t^0+b_t^2+c_t^1,2b_t^2+c_t^1,b_t^2+2c_t^1\}$. By induction hypothesis, we have
\begin{equation}\label{Fgamma01}
b_{t+1}^2 = 2c_t^1 + b_t^2.
\end{equation}
Analogously, we  obtain the following relations:
\begin{equation}\label{Fgamma02}
c_{t+1}^1 = 3c_t^1,
\end{equation}
\begin{equation}\label{Fgamma03}
d_{t+1}^0 = 3d_t^0,
\end{equation}
\begin{equation}\label{Fgamma04}
c_{t+1}^0 = 2c_t^0 + c_t^1.
\end{equation}
Comparing the above-obtained  Eqs.~(\ref{Fgamma01}-\ref{Fgamma04}) and using the induction hypothesis $c_t^0 > b_t^2 > c_t^1 = d_t^0$ yields $c_{t+1}^0 > b_{t+1}^2 > c_{t+1}^1 = d_{t+1}^0$. Therefore, the lemma is true for $n=t+1$.

This completes the proof.
\end{proof}

%Using two lemmas above, we can easily calculate the domination number of pseudofractal scale-free network.
\begin{theorem}\label{SFdomN}
The edge domination number of network $\mathcal{G}_n$, $n\geq3$, is
\begin{equation}\label{Fgamma04x}
\gamma_n = 3^{n-2}.
\end{equation}
\end{theorem}
\begin{proof}
Lemma~\ref{edgDom02}, together with Eq.~\eqref{Fgamma03} leads to
\begin{equation}\label{Fgamma05}
\gamma_{n+1} = d_{n+1}^0 = 3d_n^0 = 3\gamma_n ,
\end{equation}
which, under the initial condition $\gamma_{3}=3$, is solved to give the result.
\end{proof}

Theorem~\ref{SFdomN}, particularly Eq.~\eqref{Fgamma05}, means that for each MEDS $\chi$ of $\mathcal{G}_n$,   the three hub vertices are dominated by edges in $\chi$.  This can be easily understood. Since  hub vertices are incident to  many edges, if any of its incident edges is included in an EDS, other incident edges are excluded.
\begin{lemma}
For the three  sets $\mathcal{C}_n^1$, $\mathcal{B}_n^2$ and $\mathcal{C}_n^0$, $n\geq3$, the cardinality of any  element in them is
\begin{equation}\label{Fgamma06}
c_n^1 = 3^{n-2},
\end{equation}
\begin{equation}\label{Fgamma07}
b_n^2 = 3^{n-2}+1,
\end{equation}
and
\begin{equation}\label{Fgamma08}
c_n^0 = 3^{n-2}+2^{n-2},
\end{equation}
respectively.
\end{lemma}
\begin{proof}
By Lemma~\ref{edgDom02}, we have $c_n^1=d_n^0= 3^{n-2}$, which prove Eq.~\eqref{Fgamma06}.

In an analogous way, we can prove Eqs.~\eqref{Fgamma07} and~\eqref{Fgamma08} by exploiting Lemmas~\ref{edgDom02} and~\ref{Dom06}.
\end{proof}

%\begin{theorem}
%For network $\mathcal{G}_n$,  $n \geq 3$, there is a unique minimum dominating set.
%\end{theorem}

%\begin{proof}
%Equation~\eqref{Fgamma05} and Fig.~\ref{Theta3} shows that for $n \geq 3$ any minimum dominating set of $\mathcal{G}_{n+1}$ is in fact the union of minimum dominating sets, $\Theta_n^3$, of the three replicas of $\mathcal{G}_{n}$ (i.e. $\mathcal{G}_{n}^{1}$, $\mathcal{G}_{n}^{2}$, and $\mathcal{G}_{n}^{3}$) forming $\mathcal{G}_{n+1}$, with each pair of their identified hub vertices being counted only once. Thus, any minimum dominating set of $\mathcal{G}_{n+1}$ is determined by those of $\mathcal{G}_{n}^{1}$, $\mathcal{G}_{n}^{2}$, and $\mathcal{G}_{n}^{3}$.  Since the minimum dominating set of $\mathcal{G}_3$ is unique, there is unique  minimum dominating set for $G_n$ when $n \geq 3$. Moreover, it is easy to see that the unique dominating set of $\mathcal{G}_n$, $n \geq 3$, is actually the set of all vertices of $\mathcal{G}_{n-2}$.
%\end{proof}

\subsection{The number of minimum edge dominating sets}

Let $x_n$ denote the number of MEDSs of the pseudofractal scale-free network $\mathcal{G}_{n}$, and let $y_n$ denote the number of EDSs in  $\mathcal{C}_{n}^1$.

\begin{theorem}\label{NumEDSs}
For $n\geq 3$, the two quantities $x_n$ and $y_n$ can be obtained recursively according to the following relations.
\begin{equation}\label{SFset02}
x_{n+1} = x_n^3+x_n^2y_n+2x_ny_n^2+y_n^3,
\end{equation}
\begin{equation}\label{SFset03}
y_{n+1} = x_ny_n^2+y_n^3,
\end{equation}
with the initial condition $x_3=1$ and $y_3=1$.
\end{theorem}
\begin{proof}
We first prove Eq.~\eqref{SFset02}. Since $x_n$ is in fact is the number of MEDSs for the pseudofractal scale-free web $\mathcal{G}_n$, it can be obtained by enumerating all possible configurations of MEDSs for $\mathcal{G}_n$. By using Lemma~\ref {Dom06}, Fig.~\ref {Theta1}, and the rotational symmetry of the graph, we obtain Eq.~\eqref{SFset02}.

Similarly, we can prove the Eq.~\eqref {SFset03}.
\end{proof}

%We can know from Eqs.~\eqref{SFset02} and Eqs.~\eqref{SFset03} that the number of the edge domination sets grows exponentially with $n$.
%Since the dominating scheme for $S_{n+1}$ is given by three copies of $S_n$, we can obtain the recursive equation if we enumerate all possible configurations of minimum domination set for $S_n$.

%Applying Eqs.~\eqref{SFset02}-\eqref{SFset03}, the values of $x_n$, $y_n$, $z_n$ and $w_n$ can be recursively obtained for small $n$ as listed in Table~\ref{SetNoo}, which shows that these quantities grow exponentially with $n$.

\section{EDGE DOMINATION NUMBER AND THE NUMBER OF MINIMUM EDGE DOMINATING SETS IN SIERPI\'NSKI GRAPH}

In this section, we address the edge domination number and the number of MEDSs in the Sierpi\'nski graph.

\subsection{Construction of Sierpi\'nski graph}

The Sierpinski graph is also created in an iterative approach. We use  $\mathcal{S}_n$, $n\geq 1$, to represent the $n$-generation graph. Then the Sierpi\'nski graph is generated as follows. Initially ($n=1$), $\mathcal{S}_1$ is an equilateral triangle including three vertices and three edges. For $n =2$, to obtain $\mathcal{S}_2$, we perform a bisection of the three sides of  $\mathcal{S}_1$  yielding four small replicas of the original equilateral triangle, and remove the central downward pointing triangle. For $n>2$, $\mathcal{S}_n$ is obtained from $\mathcal{S}_{n-1}$ by performing the two bisecting and removing operations for each triangle in $\mathcal{S}_{n-1}$. Figure~\ref{Sierp01} shows the first three generations of Sierpi\'nski graph, $\mathcal{S}_1$, $\mathcal{S}_2$ and $\mathcal{S}_3$.

%%%%%%%%%%%%%%%%%%%%%%%%%%%%%%%%%%%%%%%%%%%%%%%%%%%%%%%%%
% Figure  8
%%%%%%%%%%%%%%%%%%%%%%%%%%%%%%%%%%%%%%%%%%%%%%%%%%%%%%%%%%
%\begin{figure}[htbp]
%%\begin{center}
%%\includegraphics[width=0.60\textwidth]{SGnet.eps} %[width=.60\linewidth,trim=60 40 60 40]
%%\end{center}
%\includegraphics{SGnet.eps}
%\caption[kurzform]{The first three generations of the Sierpi\'nski graph.} \label{Sierp01}
%\end{figure}

\begin{figurehere}
\centerline{
  \includegraphics[width=0.30\textwidth]{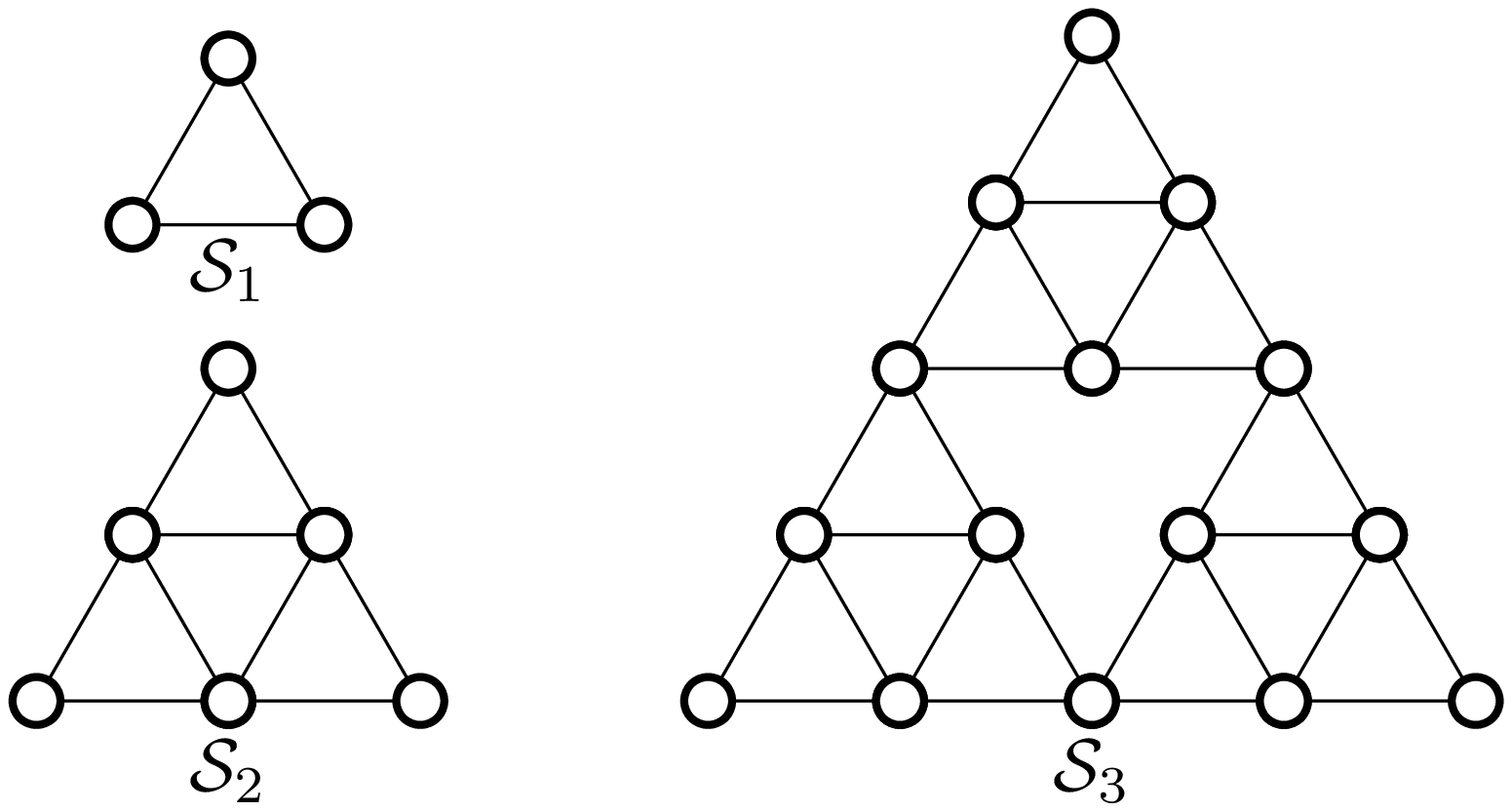}
}
\caption[kurzform]{The first three generations of the Sierpi\'nski graph.} \label{Sierp01}
\end{figurehere}
%%%%%%%%%%%%%%%%%%%%%%%%%%%%%%%%%%%%%%%%%%%%%%%%%%%%%%%%%%

By construction,  both the number of vertices and the number of edges in the Sierpi\'nskigraph $\mathcal{S}_n$ are the same as  those for the pseudofractal scale-free web $\mathcal{G}_n$, which are $N_n=(3^{n}+3)/2$ and $E_n=3^{n}$, respectively.

Distinct from $\mathcal{G}_n$, the Sierpi\'nski graph $\mathcal{S}_n$ is homogeneous with the degree of their vertices being  3, excluding the topmost vertex $A_n$, leftmost vertex $B_n$ and the rightmost vertex $C_n$, whose degree  is 2. We call these three vertices  as outmost vertices.

%%%%%%%%%%%%%%%%%%%%%%%%%%%%%%%%%%%%%%%%%%%%%%%%%%%%%%%%%
% Figure  9
%%%%%%%%%%%%%%%%%%%%%%%%%%%%%%%%%%%%%%%%%%%%%%%%%%%%%%%%%%
%\begin{figure*}
%\begin{center}
%\includegraphics[width=8.5cm]{SGsimilar.eps}
%\caption{Alternative construction of the Sierpi\'nski graph.} \label{merge}
%\end{center}
%\end{figure*}

\begin{figurehere}
\centerline{
  \includegraphics[width=8cm]{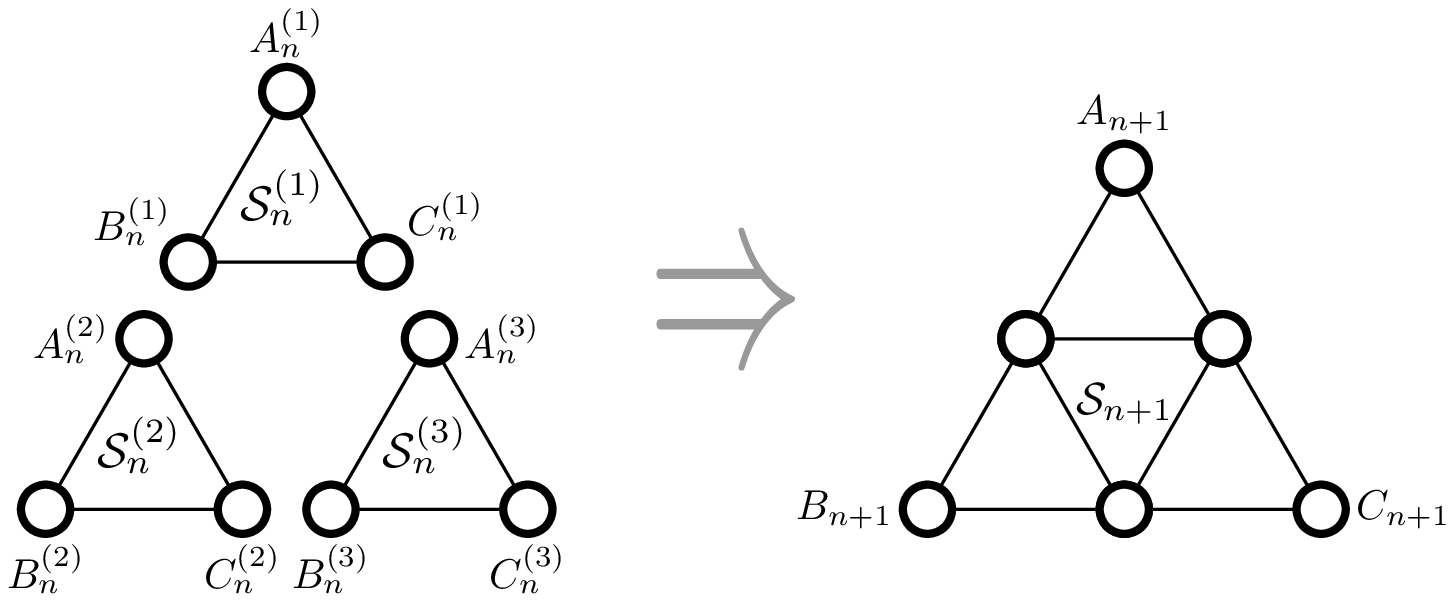}
}
\caption{Alternative construction of the Sierpi\'nski graph.} \label{merge}
\end{figurehere}
%%%%%%%%%%%%%%%%%%%%%%%%%%%%%%%%%%%%%%%%%%%%%%%%%%%%%%%%%%

In addition the number of vertices and edges, another  similarity between the pseudofractal scale-free web and the Sierpi\'nski graph is  that the latter is also self-similar, which allows us to construct the  Sierpi\'nski graph in an alternative way highlighting its self-similarity. Given the $n$th generation Sierpi\'nski graph $\mathcal{S}_{n}$,  $\mathcal{S}_{n+1}$ can be obtained by merging three copies of $\mathcal{S}_{n}$ at their outmost vertices, see Fig.~\ref{merge}. Let $\mathcal{S}_{n}^{(\theta)}$, $\theta=1,2,3$, denote three copies of $\mathcal{S}_{n}$, and let $A_{n}^{(\theta)}$, $B_{n}^{(\theta)}$, and $C_{n}^{(\theta)}$ represent, respectively, their outmost vertices. Then, one can get $\mathcal{S}_{n+1}$   by coalescing $\mathcal{S}_{n}^{(\theta)}$, with $A_{n}^{(1)}$, $B_{n}^{(2)}$, and $C_{n}^{(3)}$ being the outmost vertices $A_{n+1}$, $B_{n+1}$, and $C_{n+1}$  of $\mathcal{S}_{n+1}$.

\subsection{Edge domination number}

In the case without inducing confusion, for the Sierpi\'nski graph $\mathcal{S}_n$ we employ the same notation as those for pseudofractal scale-free web $\mathcal{G}_n$ considered in the previous section.

For an arbitrary EDS $\chi$ of $\mathcal{G}_n$, there exist three possible states for each of three outmost vertices  according to whether its incident edges are in  $\chi$ or not. For the first state, at least one incident edge is in $\chi$, we denote this state by a filled outmost vertex $\bullet$. For the second state, neither of its incident edges  belongs to  $\chi$,  but both are adjacent to some edges in $\chi$.  We denote this state by an empty outmost vertex $\bigcirc$. While for the third state, neither of its incident edges is in  $\chi$,  but at least one  incident edge is not dominated by  an edge in $\chi$.  We denote this state by a cross outmost vertex $\otimes$.

Let $\gamma_n$ be the edge domination number of  $\mathcal{S}_n$. In order to determine $\gamma_n$, we introduce some more quantities.  First, according to the state of the outmost vertices, all the EDSs of $\mathcal{S}_n$ can be classified   into four classes: $\mathcal{A}_n$, $\mathcal{B}_n$, $\mathcal{C}_n$ and $\mathcal{D}_n$. For each EDS in $\mathcal{A}_n$, there is no filled  outmost vertex.  For each EDS in  $\mathcal{B}_n$, there is one  and only one  filled outmost vertex. While for each EDS in $\mathcal{C}_n$ ($\mathcal{D}_n$),  there are exactly two (three) filled outmost vertices. Moreover,  for $\mathcal{A}_n$, $\mathcal{B}_n$, $\mathcal{C}_n$ and $\mathcal{D}_n$, we can further define some subsets of them  with the smallest cardinality: $\mathcal{A}_n^i$, $\mathcal{B}_n^i$, $\mathcal{C}_n^i$ and $\mathcal{D}_n^i$, $i\in \{ 0, 1, 2,3\}$, where $i$ means that there is/are exact $i$ cross outmost vertices in the corresponding EDS. For example, $\mathcal{B}_n^2$ denotes the subset of $\mathcal{B}_n$ such that  for each EDS in $\mathcal{B}_n^2$, it has  exactly  two cross outmost  vertices and the smallest cardinality, compared with other EDSs in $\mathcal{B}_n$. Finally,  let $a_n^i$,  $b_n^i$, $c_n^i$, and $d_n^i$,  $i=0, 1, 2, 3$,  represent  the cardinality of $\mathcal{A}_n^i$, $\mathcal{B}_n^i$, $\mathcal{C}_n^i$, and $\mathcal{D}_n^i$, respectively.  By definition, we have the following lemma.

\begin{lemma}\label{leSGDom01}
The edge domination number of the Sierpi\'nski graph $\mathcal{S}_n$, $n\geq 3$, is $\gamma_n = \min\{a_n^0,b_n^0,c_n^0,d_n^0\}$.
\end{lemma}

%%%%%%%%%%%%%%%%%%%%%%%%%%%%%%%%%%%%%%%%%%%%%%%%%%%%%%%%%
% Figure 10
%%%%%%%%%%%%%%%%%%%%%%%%%%%%%%%%%%%%%%%%%%%%%%%%%%%%%%%%%%
%\begin{figure}
%\begin{center}
%\includegraphics[width=0.70\linewidth,trim=0 80 0 500]{DemoDomiSet.eps}
%\end{center}
%\caption[kurzform]{\label{Demo01} Illustrations for the definitions of
%$\Theta_n^k$, $k = 0,1,2,3$, only showing the three hub vertices. (a), (b), (c) and (d) correspond to a dominating set belonging to $\Theta_n^0$, $\Theta_n^1$, $\Theta_n^2$, and $\Theta_n^3$, respectively. Empty circle denotes a hub  in the dominating set, while filled circle represents a hub not in the dominating set.}
%\end{figure}
%%%%%%%%%%%%%%%%%%%%%%%%%%%%%%%%%%%%%%%%%%%%%%%%%%%%%%%%%%

Thus,  to evaluate $\gamma_n$ for $\mathcal{S}_n$, we can alternatively determine  $a_n^0$, $b_n^0$, $c_n^0$, $d_n^0$.  By definition, for $a_n^i$,  $b_n^i$, $c_n^i$, and $d_n^i$, $i=0, 1, 2, 3$, only the following ten quantities are nonzero, that is $a_n^0$, $a_n^1$, $a_n^2$, $a_n^3$,  $b_n^0$, $b_n^1$, $b_n^2$,  $c_n^0$, $c_n^1$,  $d_n^0$. While the remaining quantities are zeros. Using the self-similarity  of the Sierpi\'nski graph, we can establish the recursion relations between these nonzero quantities as stated in the following lemma.
\begin{lemma}
\label{leSGDom02}
For the   Sierpi\'nski graph $\mathcal{S}_n$ with $n \geq 3$, the following relations hold.
\begin{small}
\begin{align}\label{SGDom01}
a_{n+1}^0 &= \min\{2a_n^0+b_n^0,2a_n^0+c_n^0,a_n^0+a_n^1+b_n^0,a_n^0+2b_n^0,\nonumber\\
&\quad a_n^0+a_n^1+c_n^0,a_n^0+b_n^0+b_n^1,a_n^0+b_n^0+c_n^0,3a_n^0, \nonumber\\
&\quad a_n^0+b_n^1+c_n^0,2a_n^1+c_n^0,a_n^1+2b_n^0,a_n^1+b_n^0+b_n^1,\nonumber\\
&\quad a_n^1+b_n^1+c_n^0,a_n^1+2c_n^0,a_n^2+2b_n^0,a_n^2+b_n^0+c_n^0,\nonumber\\
&\quad a_n^2+b_n^1+c_n^0,a_n^2+2c_n^0,3b_n^0,2b_n^0+b_n^1,2b_n^0+c_n^0,\nonumber\\
&\quad b_n^0+2b_n^1,b_n^0+b_n^1+c_n^0,b_n^0+2c_n^0,3b_n^1,2b_n^1+c_n^0,\nonumber\\
&\quad b_n^1+2c_n^0,3c_n^0,a_n^1+b_n^0+c_n^0,a_n^0+2c_n^0\},
\end{align}
\begin{align}\label{SGDom02}
a_{n+1}^1 &= \min\{2a_n^0+a_n^1,2a_n^0+b_n^1,2a_n^0+c_n^1,a_n^0+a_n^1+b_n^0,\nonumber\\
&\quad a_n^0+b_n^0+b_n^2,a_n^0+b_n^0+c_n^1,a_n^0+2b_n^1,2a_n^1+c_n^1,\nonumber\\
&\quad a_n^0+b_n^1+c_n^1,2a_n^1+b_n^0,2a_n^1+c_n^0,a_n^0+b_n^1+c_n^0,\nonumber\\
&\quad a_n^1+a_n^2+b_n^0,a_n^1+a_n^2+c_n^0,a_n^1+2b_n^0,a_n^1+2c_n^0,\nonumber\\
&\quad a_n^2+2b_n^0,a_n^2+b_n^0+b_n^1,a_n^2+b_n^0+c_n^0,a_n^1+2b_n^1,\nonumber\\
&\quad b_n^1+b_n^2+c_n^0,b_n^1+2c_n^0,b_n^1+c_n^0+c_n^1,2c_n^0+c_n^1,\nonumber\\
&\quad b_n^0+b_n^2+c_n^0,b_n^0+c_n^0+c_n^1,2b_n^1+c_n^0,b_n^2+2c_n^0,\nonumber\\
&\quad a_n^2+b_n^1+c_n^1,a_n^2+2c_n^0,a_n^3+2b_n^0,a_n^3+2c_n^0,\nonumber\\
&\quad a_n^0+a_n^1+b_n^1,a_n^0+a_n^1+c_n^0,a_n^0+a_n^1+c_n^1,\nonumber\\
&\quad a_n^0+a_n^2+b_n^0,a_n^0+a_n^2+c_n^0,a_n^0+b_n^0+b_n^1,\nonumber\\
&\quad a_n^1+b_n^0+b_n^1,a_n^1+b_n^0+c_n^0,a_n^1+b_n^0+c_n^1,\nonumber\\
&\quad a_n^1+b_n^1+c_n^0,a_n^1+b_n^1+c_n^1,a_n^2+b_n^0+c_n^1,\nonumber\\
&\quad 2b_n^0+b_n^1,2b_n^0+c_n^1,b_n^0+2b_n^1,b_n^0+b_n^1+c_n^0\},
\end{align}
\begin{align}\label{SGDom03}
a_{n+1}^2 &= \min\{a_n^0+2a_n^1,a_n^0+a_n^1+b_n^1,a_n^0+a_n^1+c_n^1,\nonumber\\
&\quad a_n^0+a_n^2+b_n^1,a_n^0+a_n^2+c_n^1,a_n^0+2b_n^1,2a_n^1+b_n^1,\nonumber\\
&\quad a_n^0+b_n^1+c_n^1,a_n^0+2c_n^1,2a_n^1+b_n^0,a_n^0+b_n^1+b_n^2,\nonumber\\
&\quad 2a_n^1+c_n^0,2a_n^1+c_n^1,a_n^1+a_n^2+b_n^0,a_n^2+b_n^2+c_n^0,\nonumber\\
&\quad a_n^1+2c_n^1,2a_n^2+c_n^0,a_n^2+b_n^0+b_n^1,a_n^2+b_n^0+b_n^2,\nonumber\\
&\quad a_n^2+b_n^0+c_n^1,a_n^2+2b_n^1,a_n^2+b_n^1+c_n^0,c_n^0+2c_n^1,\nonumber\\
&\quad a_n^2+2c_n^1,a_n^3+b_n^1+c_n^0,a_n^3+b_n^2+c_n^0,a_n^1+2b_n^1,\nonumber\\
&\quad b_n^0+2b_n^1,b_n^0+b_n^1+b_n^2,b_n^0+b_n^1+c_n^1,b_n^0+2c_n^1,\nonumber\\
&\quad 2b_n^1+c_n^0,2b_n^1+c_n^1,b_n^1+b_n^2+c_n^1,b_n^0+b_n^2+c_n^1,\nonumber\\
&\quad a_n^1+b_n^1+c_n^0,3b_n^1,a_n^1+b_n^1+c_n^1,a_n^1+b_n^2+c_n^0,\nonumber\\
&\quad a_n^1+a_n^2+c_n^0,a_n^1+a_n^2+c_n^1,a_n^1+b_n^0+b_n^1,\nonumber\\
&\quad a_n^1+b_n^0+b_n^2,a_n^1+b_n^0+c_n^1,a_n^1+b_n^1+b_n^2,\nonumber\\
&\quad b_n^1+c_n^0+c_n^1,b_n^1+2c_n^1,b_n^2+c_n^0+c_n^1\},
\end{align}
\begin{align}\label{SGDom04}
a_{n+1}^3 &= \min\{3a_n^1,2a_n^1+b_n^1,2a_n^1+c_n^1,a_n^1+a_n^2+b_n^1,3c_n^1,\nonumber\\
&\quad a_n^1+a_n^2+c_n^1,a_n^1+2b_n^1,a_n^1+b_n^1+c_n^1,a_n^2+2b_n^1,\nonumber\\
&\quad a_n^1+b_n^2+c_n^1,a_n^1+2c_n^1,2a_n^2+c_n^1,a_n^1+b_n^1+b_n^2,\nonumber\\
&\quad a_n^2+b_n^1+b_n^2,a_n^2+b_n^1+c_n^1,a_n^2+2c_n^1,b_n^2+2c_n^1,\nonumber\\
&\quad a_n^3+2b_n^1,a_n^3+b_n^1+c_n^1,a_n^3+b_n^2+c_n^1,a_n^3+2c_n^1,\nonumber\\
&\quad 3b_n^1,2b_n^1+b_n^2,2b_n^1+c_n^1,b_n^1+2b_n^2,b_n^1+b_n^2+c_n^1,\nonumber\\
&\quad b_n^1+2c_n^1,3b_n^2,2b_n^2+c_n^1,a_n^2+b_n^2+c_n^1\},
\end{align}

\begin{align}
b_{n+1}^0 &= \min\{2a_n^0+b_n^0,2a_n^0+c_n^0,2a_n^0+d_n^0,a_n^0+a_n^1+c_n^0,3c_n^0,\nonumber\\
&\quad a_n^2+b_n^0+d_n^0,a_n^2+b_n^1+d_n^0,a_n^1+b_n^1+d_n^0,2b_n^0+b_n^1,\nonumber\\
&\quad a_n^0+b_n^0+c_n^1,a_n^0+b_n^0+d_n^0,a_n^0+b_n^1+c_n^0,2c_n^0+c_n^1,\nonumber\\
&\quad 2b_n^0+b_n^2,2b_n^0+c_n^0,2b_n^0+d_n^0,b_n^0+2b_n^1,b_n^0+c_n^0+c_n^1,\nonumber
\end{align}
\begin{align}\label{SGDom05}
&\quad 2c_n^0+d_n^0,b_n^0+b_n^1+c_n^0,b_n^0+2c_n^0,b_n^2+2c_n^0,a_n^0+2c_n^0,\nonumber\\
&\quad 2a_n^1+d_n^0,a_n^1+2b_n^0,a_n^1+b_n^0+b_n^1,a_n^0+b_n^1+d_n^0,3b_n^0,\nonumber\\
&\quad a_n^1+b_n^0+c_n^0,a_n^1+b_n^0+d_n^0,a_n^1+b_n^1+c_n^0,a_n^1+2c_n^0,\nonumber\\
&\quad a_n^0+a_n^1+d_n^0,a_n^0+2b_n^0,a_n^0+b_n^0+b_n^1,a_n^0+b_n^0+c_n^0,\nonumber\\
&\quad b_n^1+2c_n^0,b_n^1+c_n^0+c_n^1,b_n^0+c_n^0+d_n^0,b_n^1+c_n^0+d_n^0\},
\end{align}

\begin{align}\label{SGDom06}
b_{n+1}^1 &= \min\{a_n^0+a_n^1+b_n^0,a_n^0+a_n^1+c_n^0,a_n^0+a_n^1+d_n^0,\nonumber\\
&\quad a_n^0+a_n^2+c_n^0,a_n^0+a_n^2+d_n^0,a_n^0+b_n^0+b_n^1,a_n^0+2b_n^1,\nonumber\\
&\quad a_n^0+b_n^0+c_n^1,a_n^0+b_n^1+c_n^0,a_n^0+b_n^1+c_n^1,2a_n^1+c_n^0,\nonumber\\
&\quad a_n^0+b_n^2+c_n^0,a_n^0+b_n^2+d_n^0,a_n^0+c_n^0+c_n^1,2a_n^1+d_n^0,\nonumber\\
&\quad a_n^1+a_n^2+d_n^0,a_n^1+2b_n^0,a_n^0+b_n^1+d_n^0,a_n^0+c_n^1+d_n^0,\nonumber\\
&\quad a_n^1+b_n^0+b_n^1,a_n^1+b_n^0+c_n^0,a_n^1+b_n^0+c_n^1,a_n^1+2c_n^1,\nonumber\\
&\quad a_n^1+b_n^1+c_n^0,a_n^1+b_n^1+c_n^1,a_n^1+b_n^1+d_n^0,a_n^2+2c_n^0,\nonumber\\
&\quad a_n^1+b_n^2+d_n^0,a_n^1+2c_n^0,a_n^1+c_n^0+c_n^1,a_n^1+c_n^0+d_n^0,\nonumber\\
&\quad a_n^1+c_n^1+d_n^0,a_n^2+2b_n^0,a_n^2+b_n^0+c_n^0,a_n^1+b_n^0+d_n^0,\nonumber\\
&\quad a_n^2+b_n^0+c_n^1,a_n^2+b_n^0+d_n^0,a_n^2+b_n^1+c_n^0,a_n^2+2c_n^1,\nonumber\\
&\quad a_n^2+b_n^2+d_n^0,a_n^2+c_n^0+c_n^1,a_n^2+c_n^0+d_n^0,2b_n^0+b_n^1,\nonumber\\
&\quad a_n^2+b_n^1+d_n^0,a_n^2+c_n^1+d_n^0,a_n^3+b_n^0+c_n^0,a_n^3+2c_n^0,\nonumber\\
&\quad a_n^3+b_n^1+d_n^0,a_n^3+c_n^0+c_n^1,a_n^3+c_n^0+d_n^0,2b_n^0+b_n^2,\nonumber\\
&\quad b_n^0+2b_n^1,b_n^0+b_n^1+b_n^2,a_n^3+b_n^0+d_n^0,a_n^1+b_n^2+c_n^0,\nonumber\\
&\quad 2b_n^0+c_n^1,b_n^0+b_n^1+c_n^0,b_n^0+b_n^1+c_n^1,b_n^0+b_n^1+d_n^0,\nonumber\\
&\quad b_n^0+b_n^2+c_n^0,b_n^0+b_n^2+c_n^1,b_n^0+b_n^2+d_n^0,b_n^0+2c_n^1,\nonumber\\
&\quad b_n^0+c_n^0+c_n^1,b_n^0+c_n^1+d_n^0,3b_n^1,2b_n^1+c_n^0,2b_n^1+c_n^1,\nonumber\\
&\quad 2b_n^1+d_n^0,b_n^1+b_n^2+c_n^0,b_n^1+b_n^2+c_n^1,b_n^1+b_n^2+d_n^0,\nonumber\\
&\quad b_n^2+c_n^0+d_n^0,b_n^1+c_n^0+c_n^1,b_n^1+c_n^0+d_n^0,b_n^1+2c_n^1,\nonumber\\
&\quad b_n^1+c_n^1+d_n^0,2b_n^2+c_n^0,b_n^2+2c_n^0,b_n^2+c_n^0+c_n^1,\nonumber\\
&\quad b_n^1+2c_n^0,2c_n^0+c_n^1,c_n^0+2c_n^1,c_n^0+c_n^1+d_n^0\},
\end{align}
\begin{align}\label{SGDom07}
b_{n+1}^2 &= \min\{2a_n^1+b_n^0,2a_n^1+c_n^0,2a_n^1+d_n^0,a_n^1+a_n^2+c_n^0,\nonumber\\
&\quad a_n^1+a_n^2+d_n^0,a_n^1+b_n^0+b_n^1,a_n^1+b_n^0+c_n^1,a_n^1+2b_n^1,\nonumber\\
&\quad a_n^1+b_n^1+c_n^0,a_n^1+b_n^1+c_n^1,a_n^1+b_n^1+d_n^0,2a_n^2+d_n^0,\nonumber\\
&\quad a_n^1+b_n^2+d_n^0,a_n^1+c_n^0+c_n^1,a_n^2+b_n^0+b_n^1,b_n^0+2b_n^1,\nonumber\\
&\quad a_n^2+b_n^0+c_n^1,a_n^2+b_n^1+c_n^0,a_n^2+b_n^1+c_n^1,2c_n^1+d_n^0,\nonumber\\
&\quad a_n^2+b_n^1+d_n^0,a_n^2+b_n^2+c_n^0,a_n^2+b_n^2+d_n^0,b_n^0+2c_n^1,\nonumber\\
&\quad a_n^3+b_n^1+d_n^0,a_n^3+b_n^2+d_n^0,b_n^0+b_n^1+b_n^2,2b_n^1+b_n^2,\nonumber\\
&\quad a_n^2+c_n^0+c_n^1,a_n^1+b_n^2+c_n^0,b_n^0+b_n^1+c_n^1,b_n^1+2c_n^1,\nonumber\\
&\quad 2b_n^1+c_n^0,2b_n^1+c_n^1,2b_n^1+d_n^0,b_n^1+b_n^2+c_n^0,3b_n^1,3c_n^1,\nonumber\\
&\quad b_n^1+c_n^0+c_n^1,b_n^1+c_n^1+d_n^0,b_n^2+c_n^0+c_n^1,b_n^2+2c_n^1,\nonumber\\
&\quad b_n^2+c_n^1+d_n^0,c_n^0+2c_n^1\},
\end{align}
\begin{align}\label{SGDom08}
c_{n+1}^0 &= \min\{a_n^0+2b_n^0,a_n^0+b_n^0+c_n^0,a_n^0+b_n^0+d_n^0,3b_n^0,\nonumber\\
&\quad a_n^0+b_n^1+c_n^0,a_n^0+b_n^1+d_n^0,a_n^0+2c_n^0,a_n^0+c_n^0+c_n^1,\nonumber\\
&\quad a_n^0+c_n^0+d_n^0,a_n^0+2d_n^0,a_n^1+b_n^0+c_n^0,a_n^1+b_n^0+d_n^0,\nonumber\\
&\quad a_n^1+2c_n^0,a_n^1+c_n^0+c_n^1,a_n^1+c_n^0+d_n^0,a_n^1+b_n^1+d_n^0,\nonumber\\
&\quad a_n^1+2d_n^0,a_n^2+2c_n^0,a_n^2+2d_n^0,2b_n^0+c_n^0,b_n^0+c_n^0+d_n^0,\nonumber\\
&\quad 2b_n^0+c_n^1,2b_n^0+d_n^0,b_n^0+b_n^1+c_n^0,b_n^0+2c_n^0,c_n^0+2d_n^0,\nonumber\\
&\quad b_n^2+2c_n^0,b_n^2+c_n^0+c_n^1,3c_n^0,2c_n^0+d_n^0,c_n^0+c_n^1+d_n^0,\nonumber\\
&\quad b_n^0+b_n^1+c_n^1,b_n^0+b_n^1+d_n^0,b_n^0+c_n^0+c_n^1,2b_n^0+b_n^1,\nonumber\\
&\quad b_n^1+c_n^0+c_n^1,b_n^1+c_n^0+d_n^0,b_n^1+c_n^1+d_n^0,b_n^1+2d_n^0,\nonumber\\
&\quad b_n^0+c_n^1+d_n^0,b_n^0+2d_n^0,2b_n^1+c_n^0,b_n^1+2c_n^0\},
\end{align}
\begin{align}\label{SGDom09}
c_{n+1}^1 &= \min\{a_n^1+2b_n^0,a_n^1+b_n^0+c_n^0,a_n^1+b_n^0+d_n^0,\nonumber\\
&\quad a_n^1+b_n^1+c_n^0,a_n^1+b_n^1+d_n^0,a_n^1+2c_n^0,a_n^1+c_n^0+c_n^1,\nonumber\\
&\quad a_n^1+c_n^0+d_n^0,a_n^1+2d_n^0,a_n^2+b_n^0+c_n^0,a_n^2+b_n^0+d_n^0,\nonumber\\
&\quad a_n^2+b_n^1+d_n^0,a_n^2+2c_n^0,a_n^2+c_n^0+c_n^1,a_n^2+c_n^0+d_n^0,\nonumber\\
&\quad b_n^0+2b_n^1,b_n^0+b_n^1+c_n^0,b_n^0+b_n^1+c_n^1,b_n^1+c_n^0+c_n^1,\nonumber\\
&\quad b_n^0+b_n^2+c_n^0,b_n^0+c_n^0+c_n^1,b_n^0+2c_n^1,c_n^0+c_n^1+d_n^0,\nonumber\\
&\quad b_n^1+c_n^0+d_n^0,b_n^1+2c_n^1,b_n^1+c_n^1+d_n^0,b_n^0+b_n^1+d_n^0,\nonumber\\
&\quad a_n^2+2d_n^0,a_n^3+2c_n^0,a_n^3+2d_n^0,2b_n^0+b_n^1,2b_n^0+c_n^1,\nonumber\\
&\quad 2b_n^1+c_n^1,2b_n^1+d_n^0,b_n^2+2d_n^0,b_n^1+2c_n^0,b_n^1+2d_n^0,\nonumber\\
&\quad b_n^2+2c_n^0,b_n^2+c_n^0+c_n^1,b_n^2+c_n^0+d_n^0,b_n^2+2c_n^1,\nonumber\\
&\quad b_n^2+c_n^1+d_n^0,b_n^1+b_n^2+c_n^0,2c_n^0+c_n^1,2b_n^1+c_n^0,\nonumber\\
&\quad 2c_n^1+d_n^0,c_n^1+2d_n^0\},
\end{align}
\begin{align}\label{SGDom10}
d_{n+1}^0 &= \min\{3b_n^0,2b_n^0+c_n^0,b_n^0+b_n^1+c_n^0,b_n^0+c_n^1+d_n^0,\nonumber\\
&\quad b_n^0+b_n^1+d_n^0,b_n^0+2c_n^0,b_n^0+c_n^0+c_n^1,b_n^0+c_n^0+d_n^0,\nonumber\\
&\quad b_n^0+2d_n^0,2b_n^1+d_n^0,b_n^1+2c_n^0,b_n^1+c_n^0+c_n^1,b_n^2+2d_n^0,\nonumber\\
&\quad b_n^1+c_n^0+d_n^0,b_n^1+c_n^1+d_n^0,b_n^1+2d_n^0,b_n^2+2c_n^0,3c_n^1,\nonumber\\
&\quad b_n^2+c_n^0+d_n^0,b_n^2+c_n^1+d_n^0,3c_n^0,2c_n^0+c_n^1,2c_n^0+d_n^0,\nonumber\\
&\quad c_n^0+2c_n^1,c_n^0+c_n^1+d_n^0,c_n^0+2d_n^0,2c_n^1+d_n^0,3d_n^0,\nonumber\\
&\quad c_n^1+2d_n^0,2b_n^0+d_n^0\},
\end{align}
\end{small}
\end{lemma}
\begin{proof}
This lemma can be proved graphically. Figs.~\ref{SGTheta0}-\ref{SGeta0} illustrate the graphical representations from Eq.~\eqref{SGDom01} to Eq.~\eqref{SGDom10}.
\end{proof}
%We will prove this lemma by mathematical induction on $n$. For $n = 3$, we can obtain $A_3^3 = 4$, $A_3^0 = 5,A_3^1 = 5,A_3^2 = 5,B_3^1 = 5,B_3^2 = 5,C_3^1 = 5$, and $B_3^0 = 6,C_3^0 = 6,D_3^0 = 6$, thus we get $P_3 = 4$, $Q_3 = 5$ and $R_3 = 6$.
%Assuming that the lemma hold for $n = t, t \geq 3$. Then, we take $\mathcal{A}_t+1^3$ belongs to $\mathcal{P}_t+1$ as an example. From
%\end{proof}

%%%%%%%%%%%%%%%%%%%%%%%%%%%%%%%%%%%%%%%%%%%%%%%%%%%%%%%%%
% Figure  10
%%%%%%%%%%%%%%%%%%%%%%%%%%%%%%%%%%%%%%%%%%%%%%%%%%%%%%%%%%
\begin{figure*}[htbp]
\centering
\includegraphics[width=0.8\textwidth]{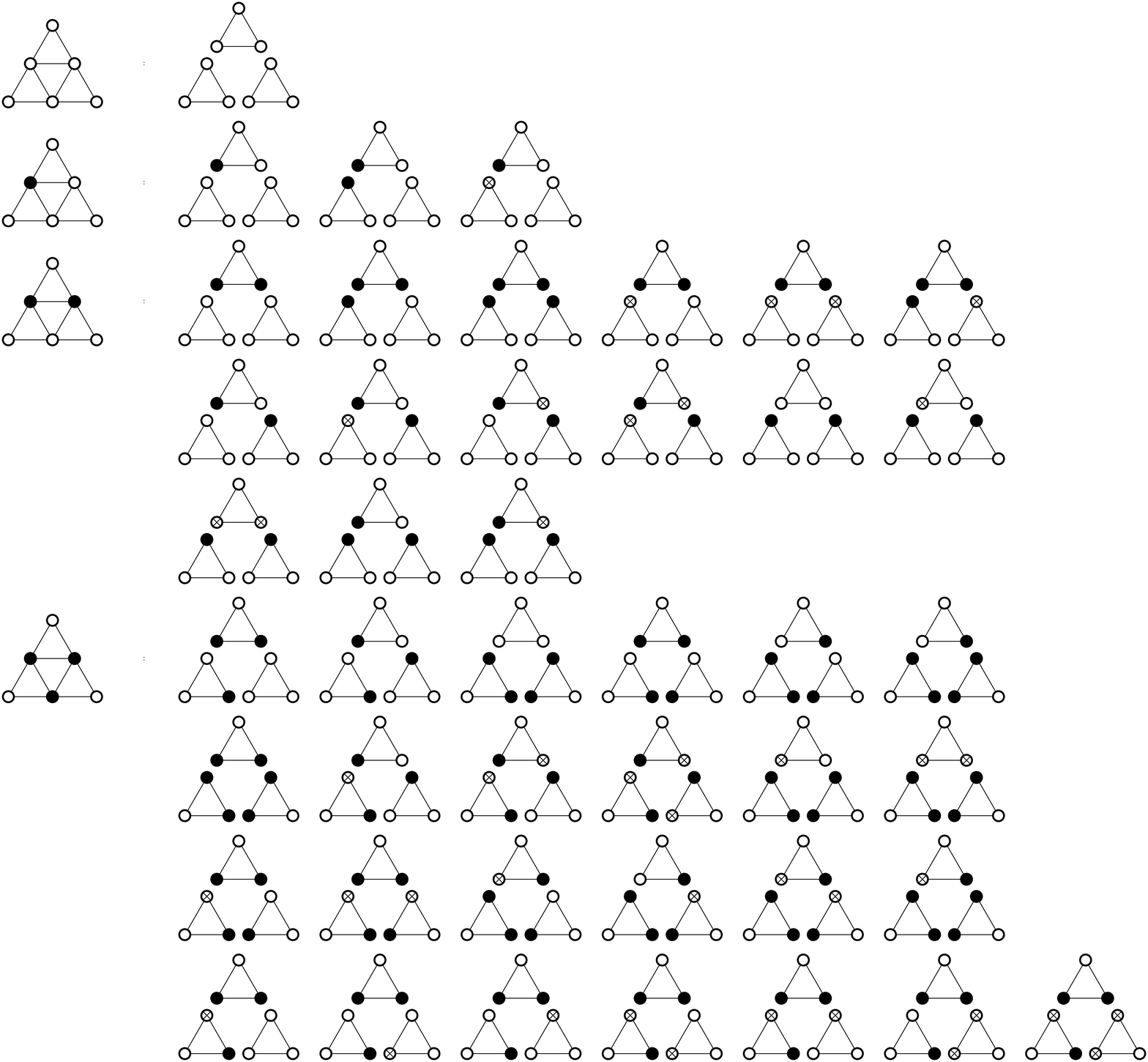}
\caption{\label{SGTheta0}Illustration of all possible configurations of EDSs in $\mathcal{A}_{n+1}^0$ for $\mathcal{S}_{n+1}$. Only the outmost vertices of $\mathcal{S}_{n}^{(\theta)}$, $\theta=1,2,3$, are shown.  It is the same with Figs.~\ref{SGTheta1}-\ref{SGeta0}. }%It is the same with Figs.~\ref{SGTheta1}-\ref{SGeta0}.
\end{figure*}
%%%%%%%%%%%%%%%%%%%%%%%%%%%%%%%%%%%%%%%%%%%%%%%%%%%%%%%%%%

%%%%%%%%%%%%%%%%%%%%%%%%%%%%%%%%%%%%%%%%%%%%%%%%%%%%%%%%%
% Figure  11
%%%%%%%%%%%%%%%%%%%%%%%%%%%%%%%%%%%%%%%%%%%%%%%%%%%%%%%%%%
\begin{figure*}[htbp]
\centering
\includegraphics[width=0.7\textwidth]{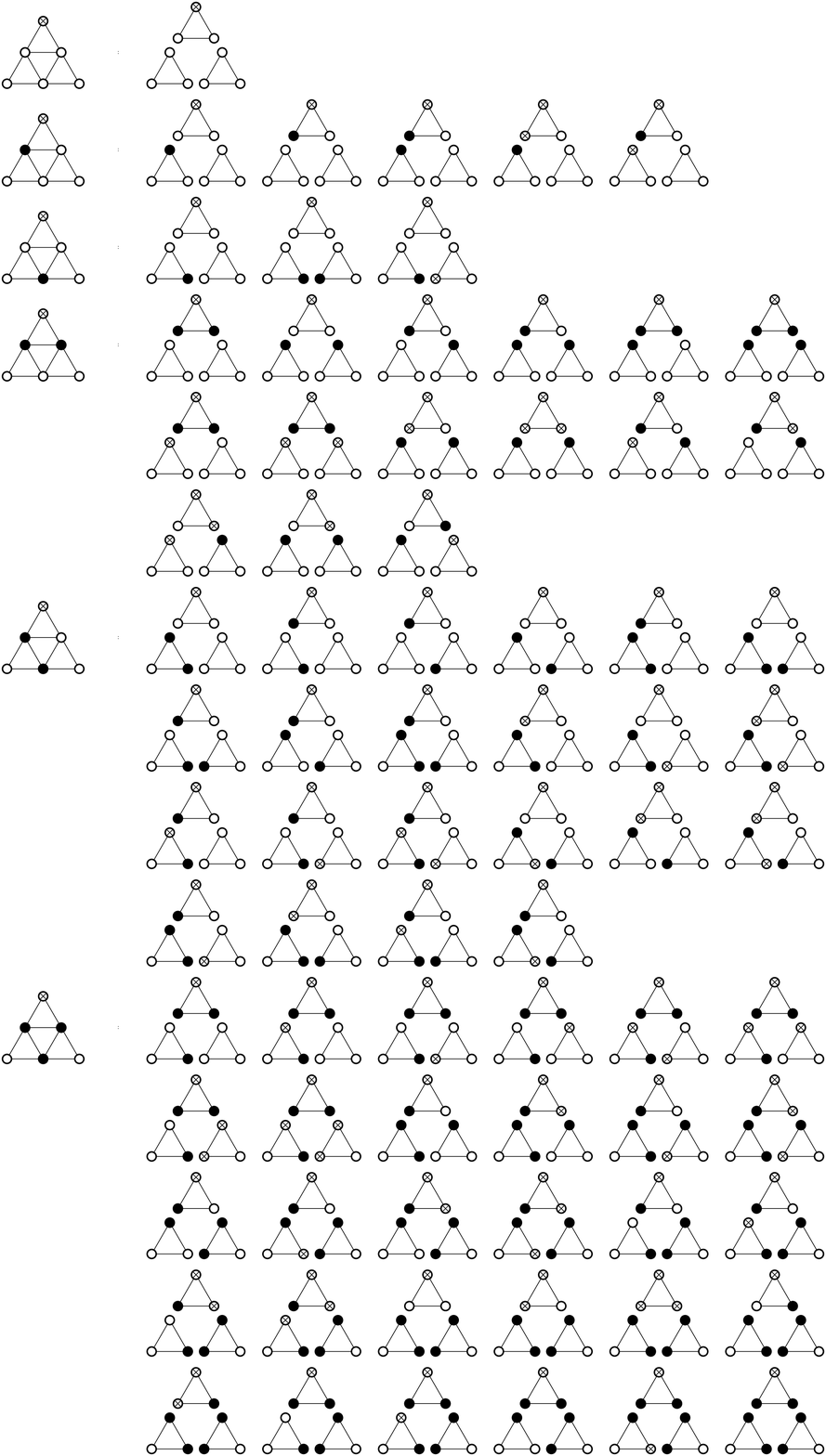}
\caption{\label{SGTheta1}Illustration of all possible configurations of EDSs in $\mathcal{A}_{n+1}^1$ for $\mathcal{S}_{n+1}$.}
\end{figure*}
%%%%%%%%%%%%%%%%%%%%%%%%%%%%%%%%%%%%%%%%%%%%%%%%%%%%%%%%%%

%%%%%%%%%%%%%%%%%%%%%%%%%%%%%%%%%%%%%%%%%%%%%%%%%%%%%%%%%
% Figure  12
%%%%%%%%%%%%%%%%%%%%%%%%%%%%%%%%%%%%%%%%%%%%%%%%%%%%%%%%%%
\begin{figure*}[htbp]
\centering
\includegraphics[width=0.65\textwidth]{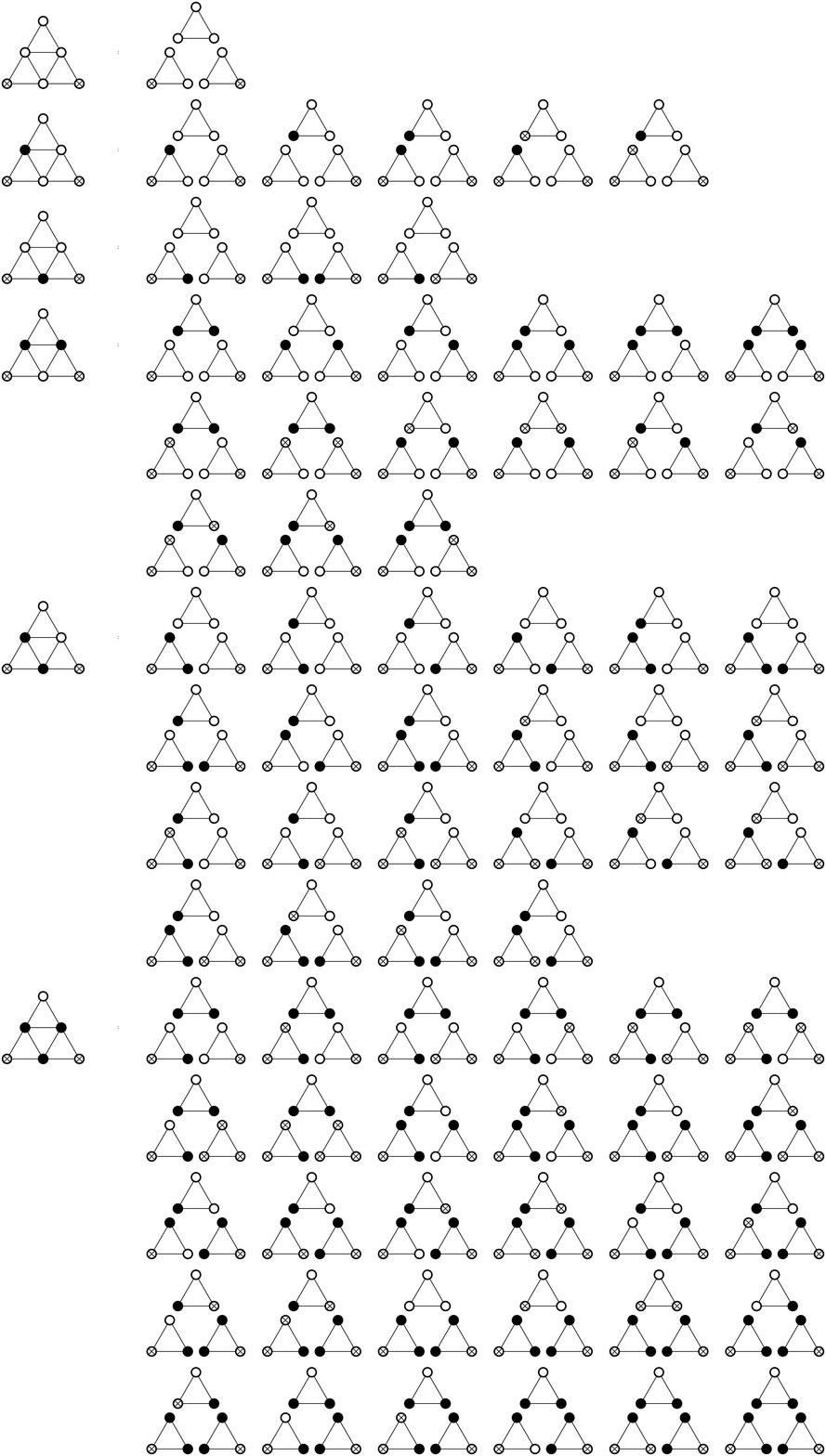}
\caption{\label{SGTheta2} Illustration of all possible configurations of EDSs in $\mathcal{A}_{n+1}^2$ for $\mathcal{S}_{n+1}$.}
\end{figure*}
%%%%%%%%%%%%%%%%%%%%%%%%%%%%%%%%%%%%%%%%%%%%%%%%%%%%%%%%%%

%%%%%%%%%%%%%%%%%%%%%%%%%%%%%%%%%%%%%%%%%%%%%%%%%%%%%%%%%%
%% Figure  13
%%%%%%%%%%%%%%%%%%%%%%%%%%%%%%%%%%%%%%%%%%%%%%%%%%%%%%%%%%%
\begin{figure*}[htbp]
\centering
\includegraphics[width=0.9\textwidth]{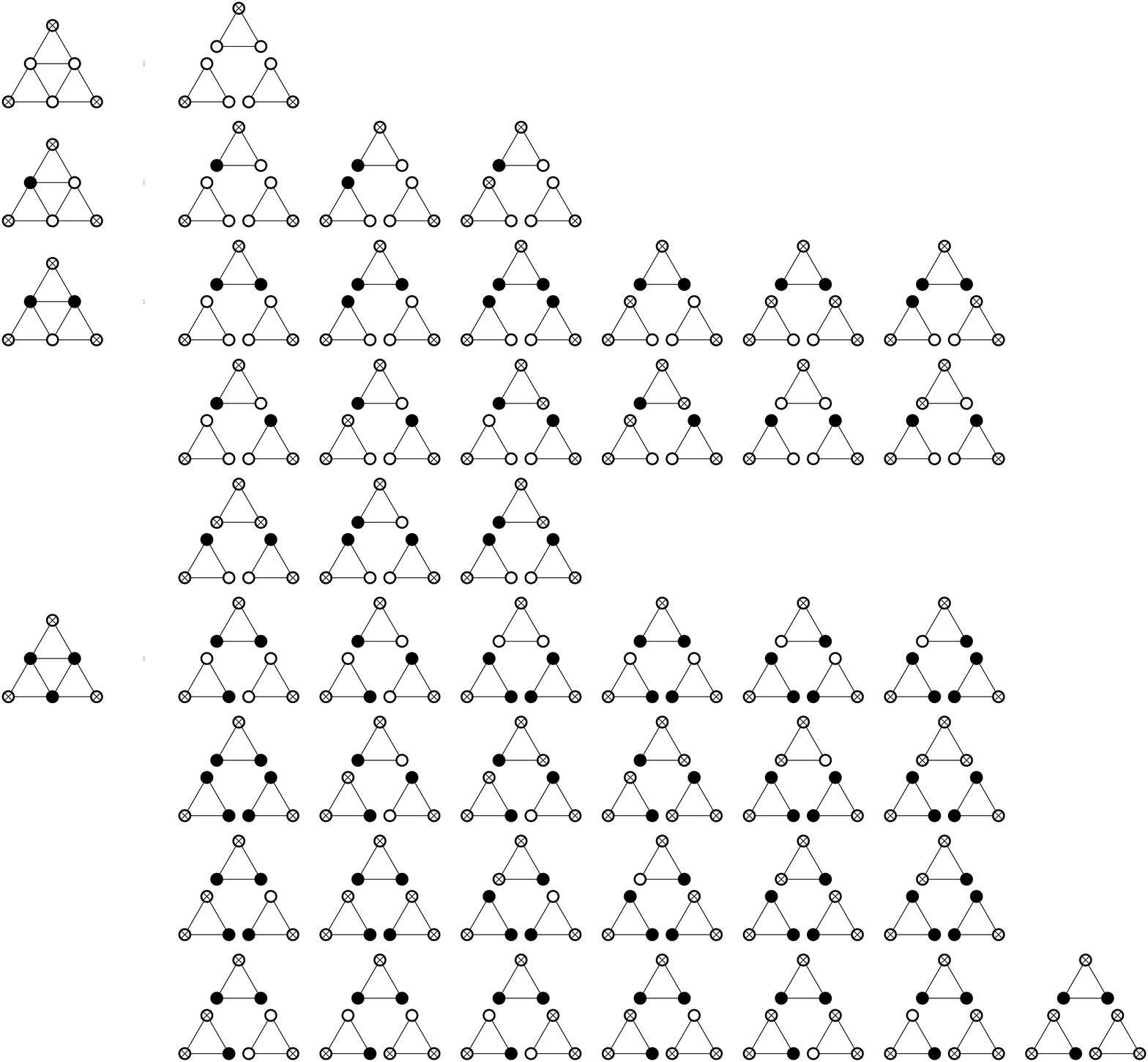}
\caption{\label{SGTheta3}Illustration of all possible configurations of EDSs in $\mathcal{A}_{n+1}^3$ for $\mathcal{S}_{n+1}$.}
\end{figure*}
%%%%%%%%%%%%%%%%%%%%%%%%%%%%%%%%%%%%%%%%%%%%%%%%%%%%%%%%%%%
%
%%%%%%%%%%%%%%%%%%%%%%%%%%%%%%%%%%%%%%%%%%%%%%%%%%%%%%%%%%
%% Figure  14
%%%%%%%%%%%%%%%%%%%%%%%%%%%%%%%%%%%%%%%%%%%%%%%%%%%%%%%%%%%
\begin{figure*}[htbp]
\centering
\includegraphics[width=0.7\textwidth]{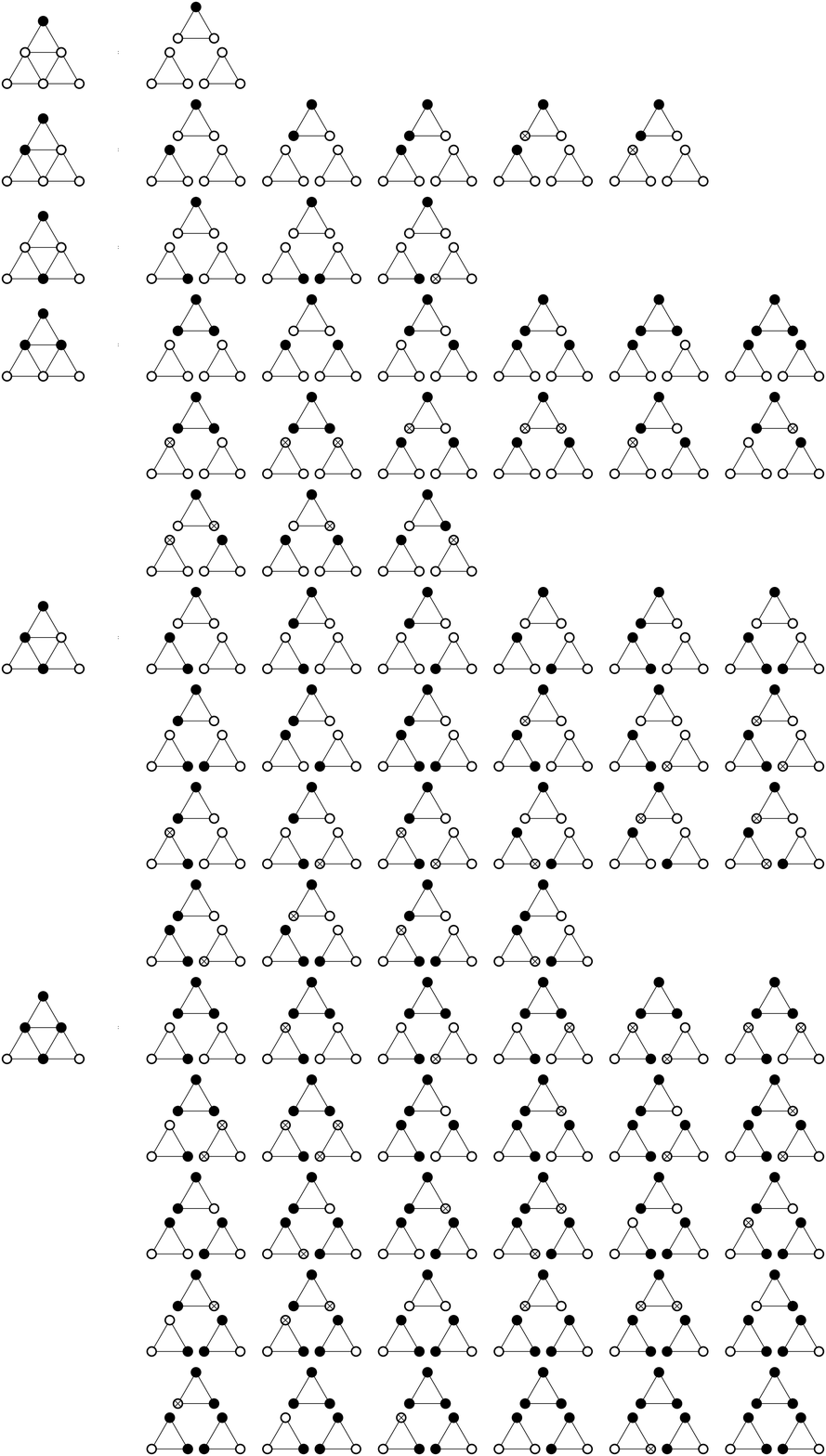}
\caption{\label{SGphi0}Illustration of all possible configurations of EDSs in $\mathcal{B}_{n+1}^0$ for $\mathcal{S}_{n+1}$.}
\end{figure*}
%%%%%%%%%%%%%%%%%%%%%%%%%%%%%%%%%%%%%%%%%%%%%%%%%%%%%%%%%%%
%
%%%%%%%%%%%%%%%%%%%%%%%%%%%%%%%%%%%%%%%%%%%%%%%%%%%%%%%%%%
%% Figure  15
%%%%%%%%%%%%%%%%%%%%%%%%%%%%%%%%%%%%%%%%%%%%%%%%%%%%%%%%%%%
\begin{figure*}[htbp]
\centering
\includegraphics[width=0.75\textwidth]{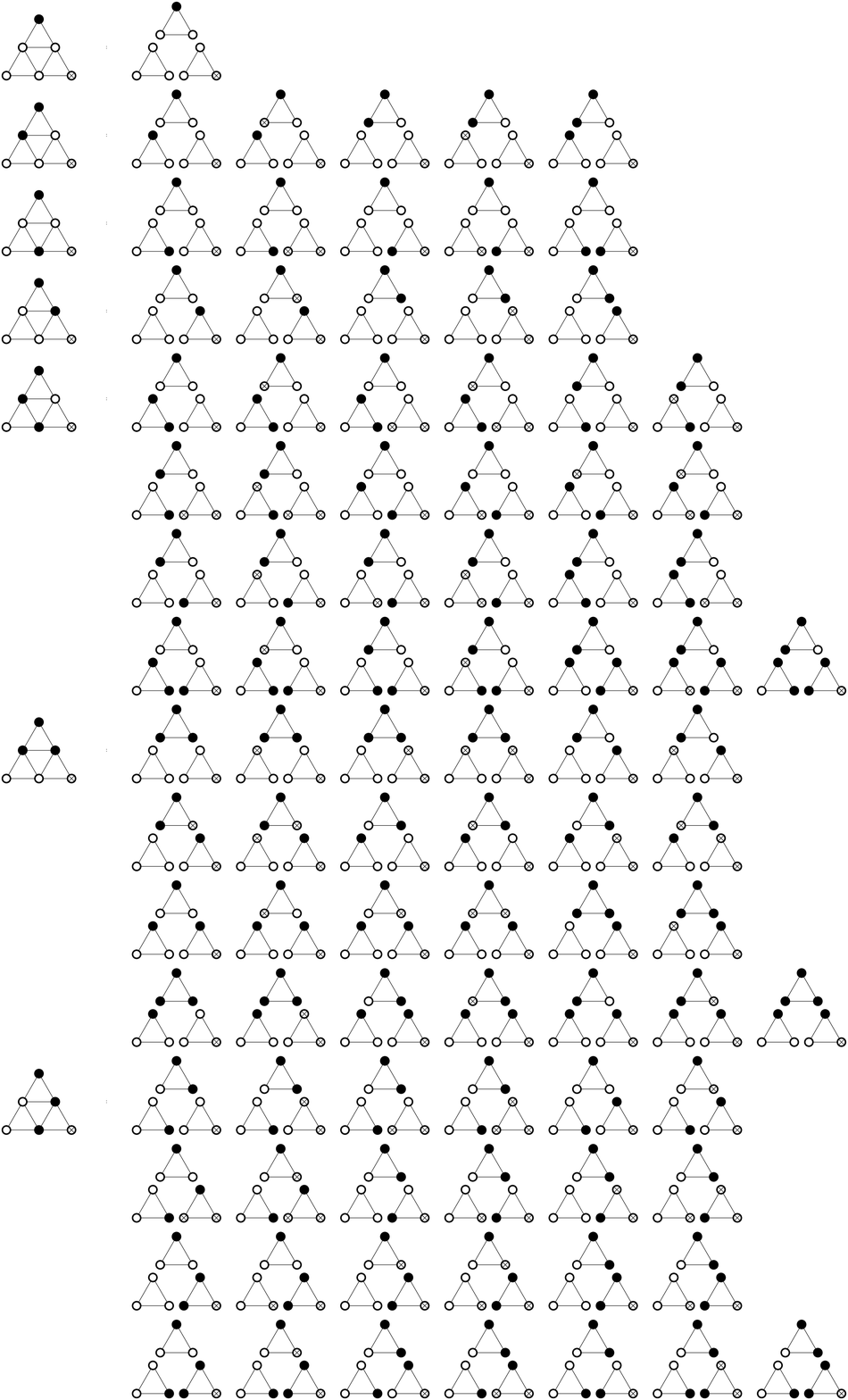}
\caption{\label{SGphi1}Illustration of all possible configurations of EDSs in $\mathcal{B}_{n+1}^1$ for $\mathcal{S}_{n+1}$.}
\end{figure*}
%%%%%%%%%%%%%%%%%%%%%%%%%%%%%%%%%%%%%%%%%%%%%%%%%%%%%%%%%%%
%
%%%%%%%%%%%%%%%%%%%%%%%%%%%%%%%%%%%%%%%%%%%%%%%%%%%%%%%%%%
%% Figure  15(continue)
%%%%%%%%%%%%%%%%%%%%%%%%%%%%%%%%%%%%%%%%%%%%%%%%%%%%%%%%%%%
\begin{figure*}[htbp]
\centering
\includegraphics[width=0.75\textwidth]{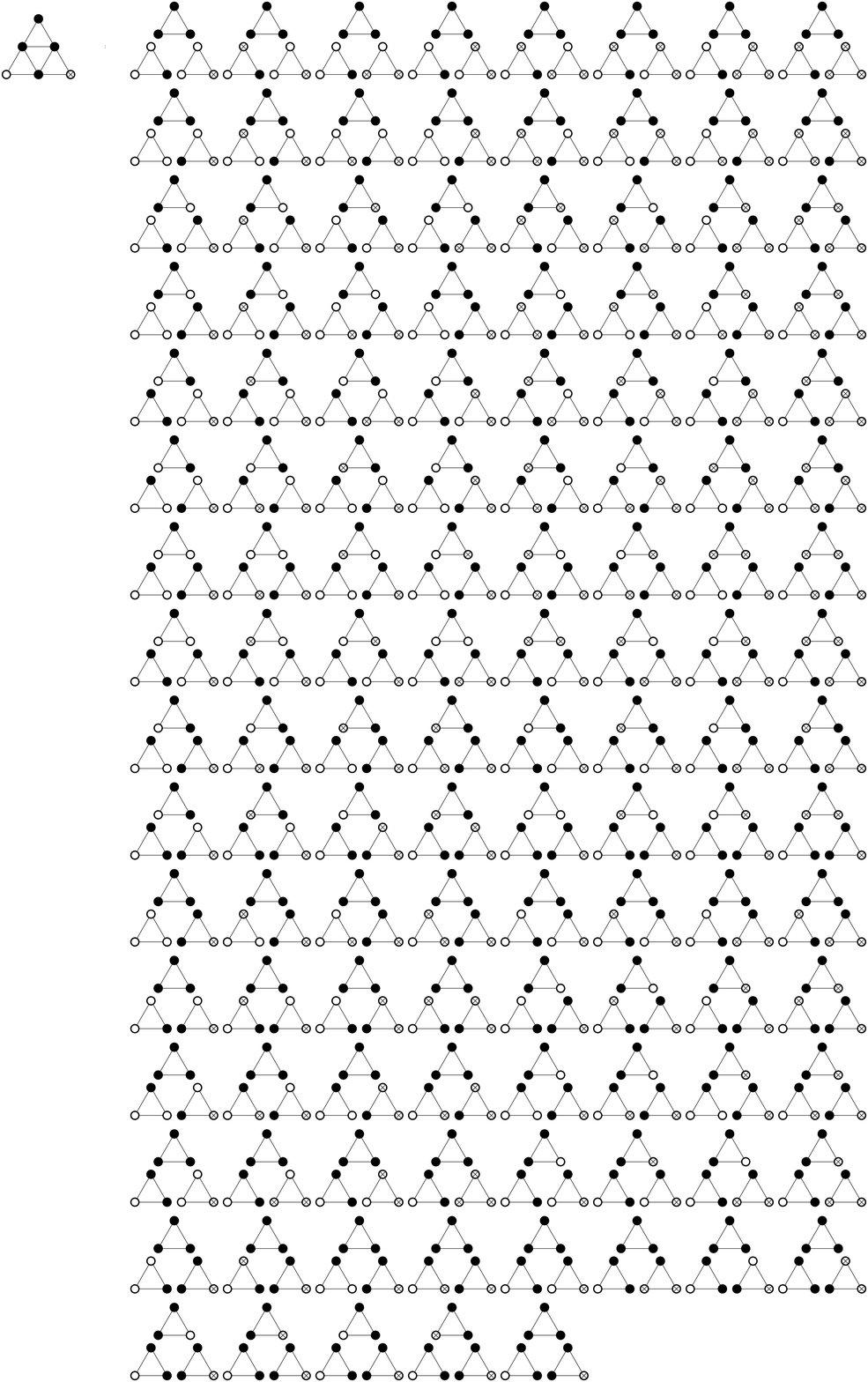}
\caption{\label{SGphi1}(continued) Illustration of all possible configurations of EDSs in $\mathcal{B}_{n+1}^1$ for $\mathcal{S}_{n+1}$.}
\end{figure*}
%%%%%%%%%%%%%%%%%%%%%%%%%%%%%%%%%%%%%%%%%%%%%%%%%%%%%%%%%%%
%
%%%%%%%%%%%%%%%%%%%%%%%%%%%%%%%%%%%%%%%%%%%%%%%%%%%%%%%%%%
%% Figure  16
%%%%%%%%%%%%%%%%%%%%%%%%%%%%%%%%%%%%%%%%%%%%%%%%%%%%%%%%%%%
\begin{figure*}[htbp]
\centering
\includegraphics[width=0.7\textwidth]{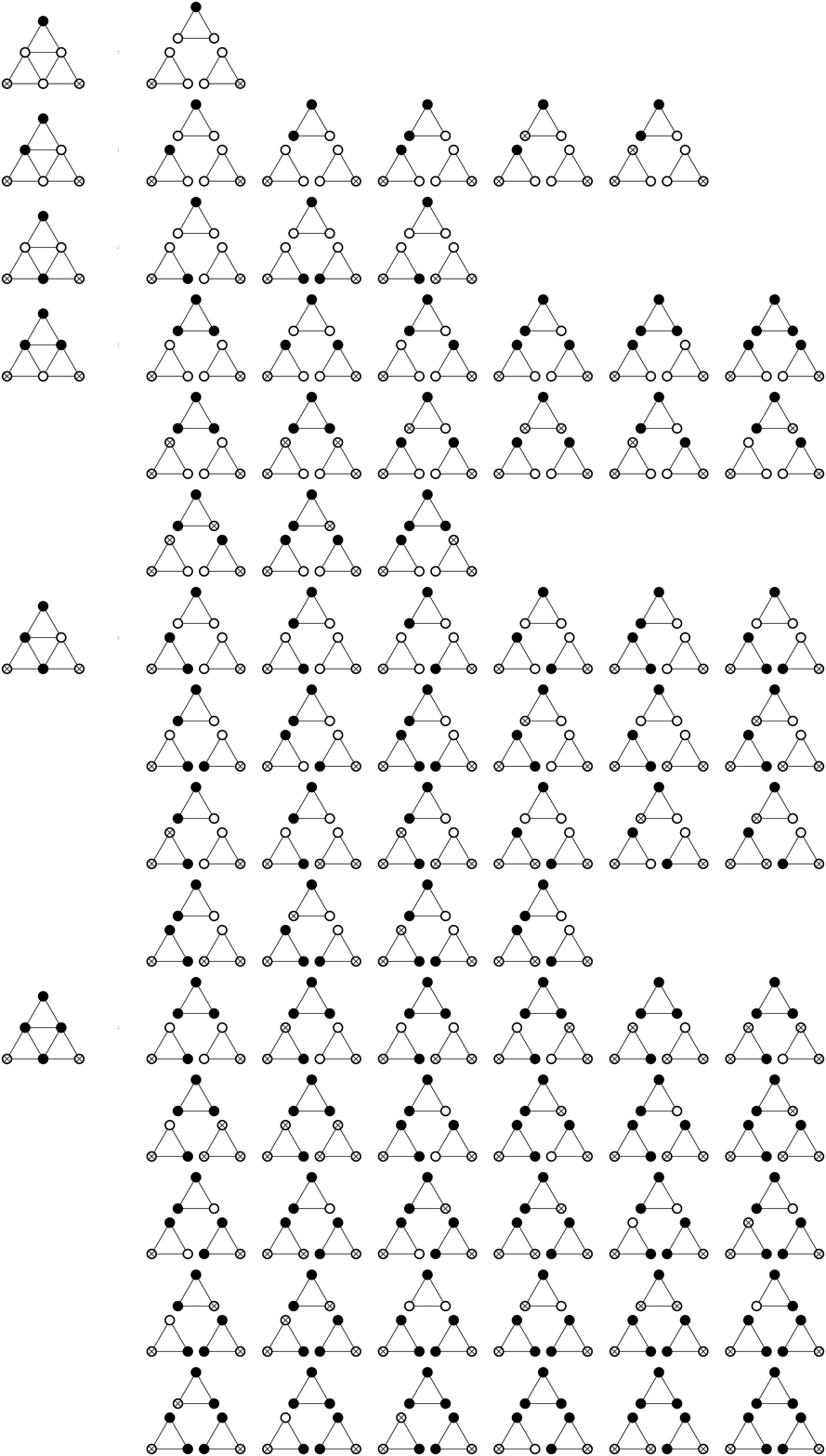}
\caption{\label{SGphi1}Illustration of all possible configurations of EDSs in $\mathcal{B}_{n+1}^2$ for $\mathcal{S}_{n+1}$.}
\end{figure*}
%%%%%%%%%%%%%%%%%%%%%%%%%%%%%%%%%%%%%%%%%%%%%%%%%%%%%%%%%%%
%
%%%%%%%%%%%%%%%%%%%%%%%%%%%%%%%%%%%%%%%%%%%%%%%%%%%%%%%%%%
%% Figure  17
%%%%%%%%%%%%%%%%%%%%%%%%%%%%%%%%%%%%%%%%%%%%%%%%%%%%%%%%%%%
\begin{figure*}[htbp]
\centering
\includegraphics[width=0.7\textwidth]{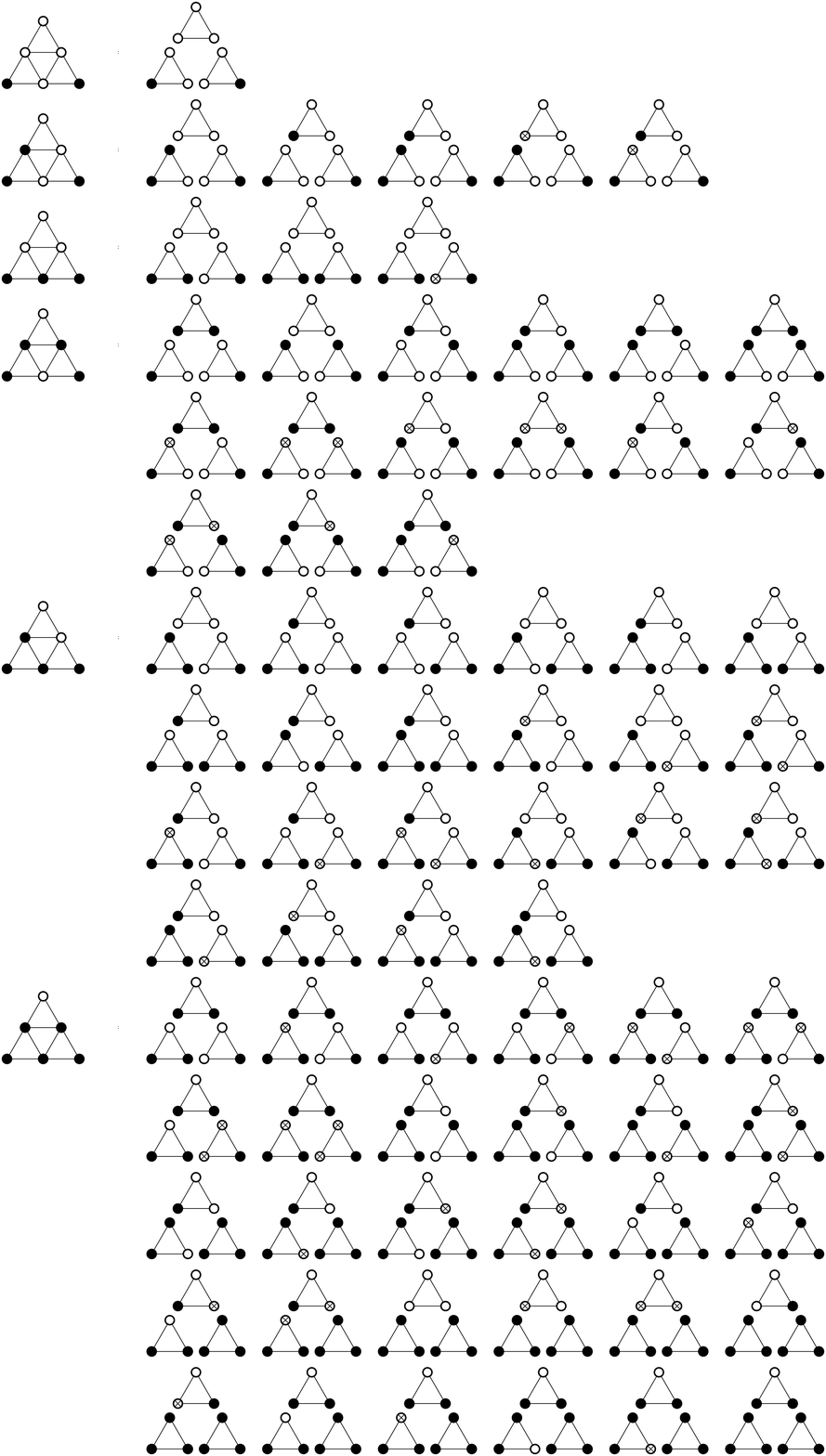}
\caption{\label{SGxi0}Illustration of all possible configurations of EDSs in $\mathcal{C}_{n+1}^0$ for $\mathcal{S}_{n+1}$.}
\end{figure*}
%%%%%%%%%%%%%%%%%%%%%%%%%%%%%%%%%%%%%%%%%%%%%%%%%%%%%%%%%%%
%
%%%%%%%%%%%%%%%%%%%%%%%%%%%%%%%%%%%%%%%%%%%%%%%%%%%%%%%%%%
%% Figure  18
%%%%%%%%%%%%%%%%%%%%%%%%%%%%%%%%%%%%%%%%%%%%%%%%%%%%%%%%%%%
\begin{figure*}[htbp]
\centering
\includegraphics[width=0.7\textwidth]{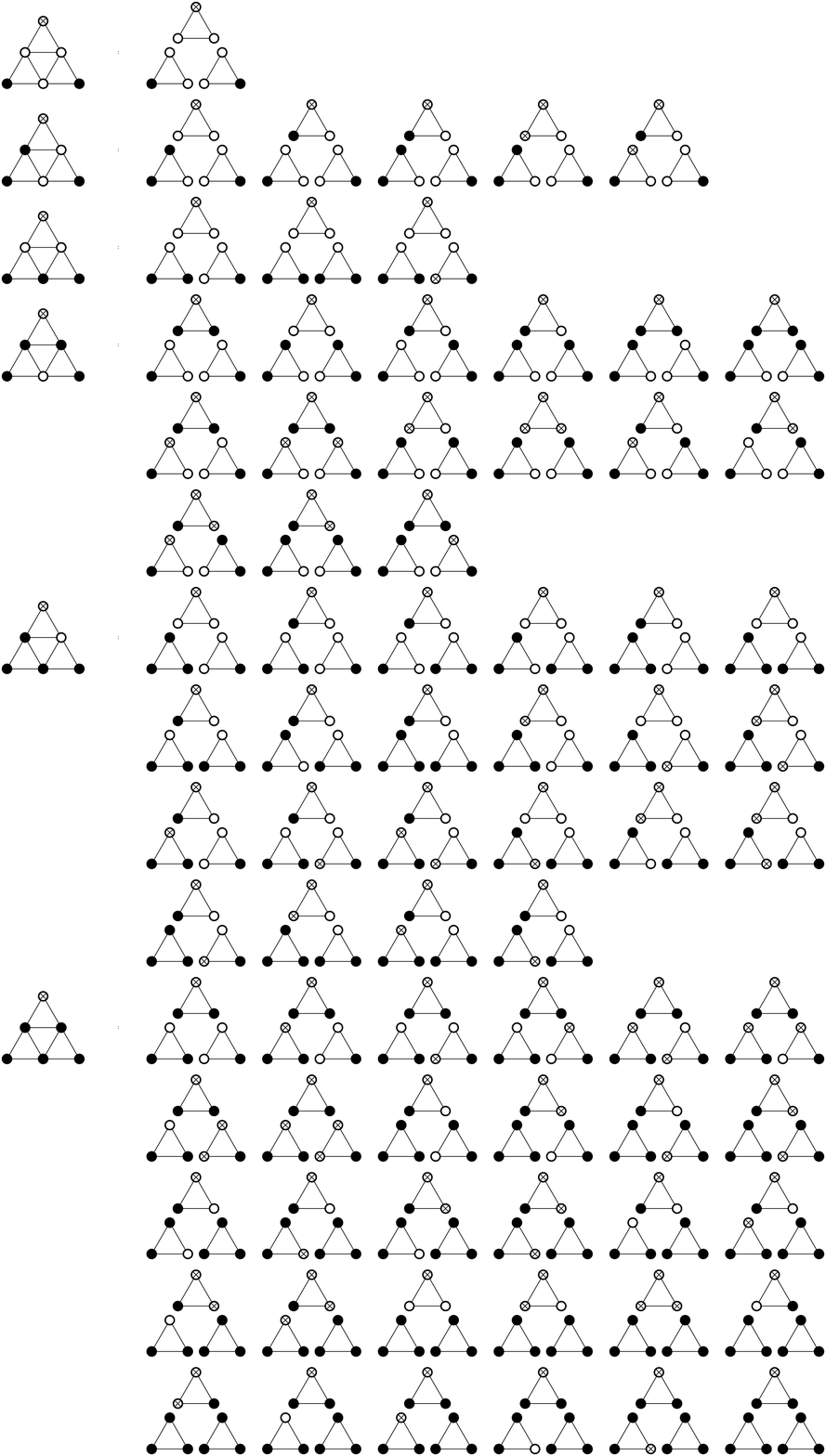}
\caption{\label{SGxi1}Illustration of all possible configurations of EDSs in $\mathcal{C}_{n+1}^1$ for $\mathcal{S}_{n+1}$.}
\end{figure*}
%%%%%%%%%%%%%%%%%%%%%%%%%%%%%%%%%%%%%%%%%%%%%%%%%%%%%%%%%%%
%
%%%%%%%%%%%%%%%%%%%%%%%%%%%%%%%%%%%%%%%%%%%%%%%%%%%%%%%%%%
%% Figure  19
%%%%%%%%%%%%%%%%%%%%%%%%%%%%%%%%%%%%%%%%%%%%%%%%%%%%%%%%%%%
\begin{figure*}[htbp]
\centering
\includegraphics[width=0.8\textwidth]{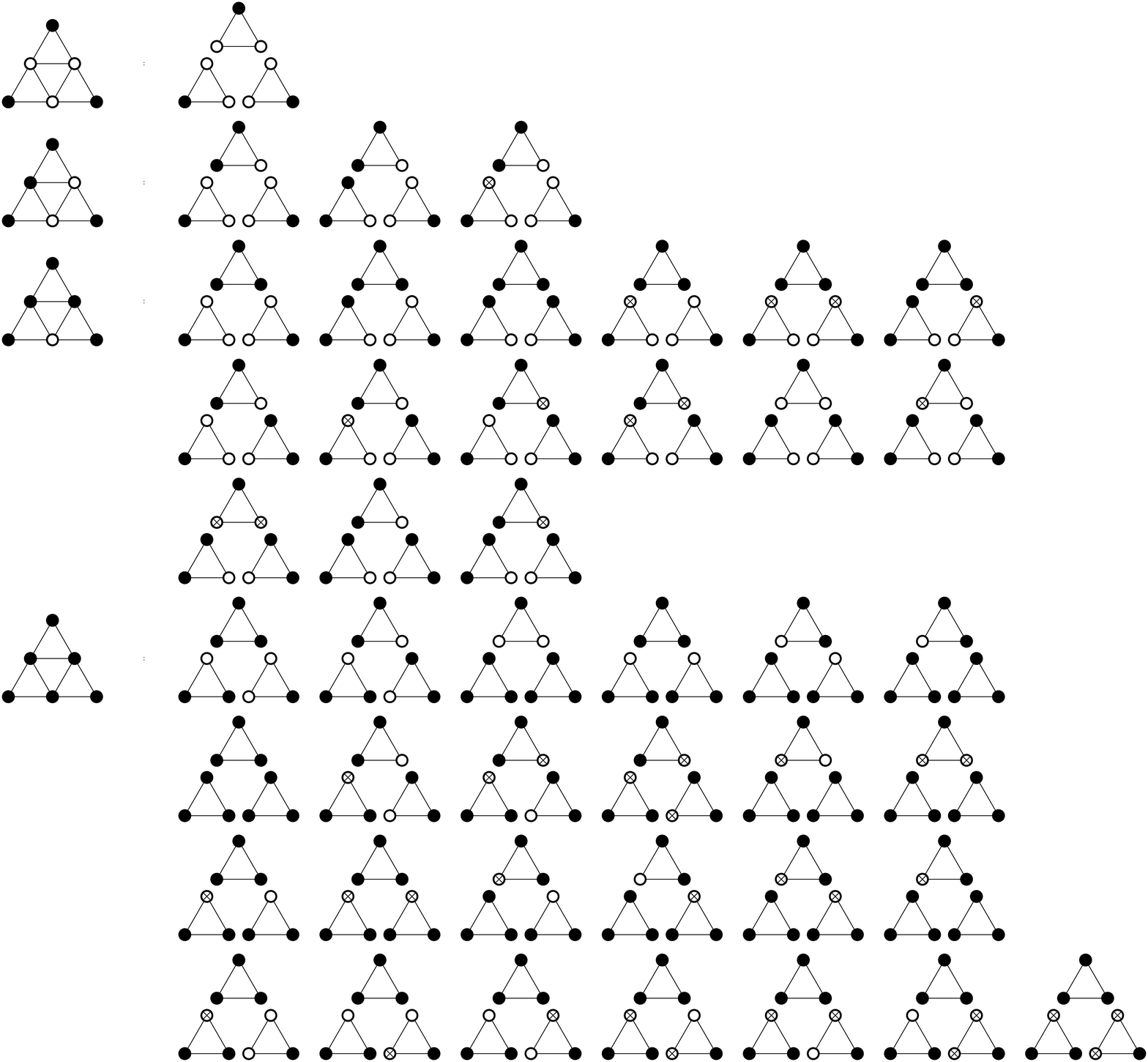}
\caption{\label{SGeta0}Illustration of all possible configurations of EDSs in $\mathcal{D}_{n+1}^0$  for $\mathcal{S}_{n+1}$.}
\end{figure*}
%%%%%%%%%%%%%%%%%%%%%%%%%%%%%%%%%%%%%%%%%%%%%%%%%%%%%%%%%%%

\begin{lemma}\label{leSGDom03}
For the   Sierpi\'nski graph $\mathcal{S}_n$ with $n \geq 3$,
\begin{align}\label{SGset01}
& a_n^3+1 = a_n^0 = a_n^1 = a_n^2 = b_n^1= b_n^2 \nonumber\\
&  = c_n^1 = b_n^0-1 = c_n^0-1 = d_n^0-1.
\end{align}
\end{lemma}
\begin{proof}
We prove this lemma by induction on $n$.

For $n = 3$, it is straightforward  to check by hand that $a_3^3 =4$, $a_3^0 = 5$, $a_3^1=5$, $a_3^2=5$, $b_3^1=5$, $b_3^2 = 5$, $c_3^1=5$, $b_3^0=6$, $c_3^0=6$, and $d_3^0=6$. Thus, the result holds for $n = 3$.

Let us assume that the lemma is true  for $n = t$. For $n =t+1$, by
induction assumption and Lemma~\ref{leSGDom02}, it is easy to check that the result is true for $n =t+1$.
\end{proof}

\begin{theorem}
The edge domination number of the Sierpi\'nski graph $\mathcal{S}_n$, $n\geq3$, is $\gamma_n = 5\cdot 3^{n-3}$.
\end{theorem}
\begin{proof}
Combining  Lemmas~\ref{leSGDom01}, \ref{leSGDom02}, and~\ref{leSGDom03}, we obtain  $\gamma_{n+1} = a_{n+1}^0 = 3a_n^0 = 3\gamma_n$. With the initial condition $\gamma_3 = 5$, we obtain $\gamma_n = 5\cdot3^{n-3}$  for all $n \geq 3$.
\end{proof}

\subsection{The number of minimum edge dominating sets}

Let $x_n$ denote the number of MEDSs in the Sierpi\'nski graph $\mathcal{S}_{n}$. Let $y_n$, $z_n$, and $w_n$  denote the number of EDSs in  $\mathcal{B}_{n}^1$, $\mathcal{C}_{n}^1$, and $\mathcal{B}_{n}^2$, respectively.

\begin{theorem}\label{NumEDSg}
For the Sierpi\'nski graph $\mathcal{S}_{n}$, $n\geq 3$, the four quantities $x_n$, $y_n$, $z_n$ and $w_n$ can be obtained by the following recursion relations.
\begin{equation}\label{SGset02}
x_{n+1} = x_n^3+y_n^3,
\end{equation}
\begin{align}\label{SGset03}
y_{n+1} =& x_ny_n^2+2x_ny_nz_n+y_n^3+\nonumber\\
&4y_n^2z_n+2y_nz_nw_n+3y_nz_n^2,
\end{align}
\begin{equation}\label{SGset04}
z_{n+1} = 2y_n^2z_n+2y_nz_n^2+z_n^2w_n,
\end{equation}
\begin{align}\label{SGset05}
w_{n+1} = & 2y_n^3+2y_n^2w_n+2y_n^2z_n+\nonumber\\
&4y_nz_n^2+4w_nz_n^2+z_n^3,
\end{align}
with the initial condition $x_3=2$, $y_3=9$, $z_3=4$, and $w_3=16$.
\end{theorem}
\begin{proof}
We first prove Eq.~\eqref{SGset02}. By definition, $x_n$ is in fact the number of all different MEDSs for the Sierpi\'nski graph $\mathcal{S}_n$. One can determine $x_n$ via enumerating all possible configurations of MEDSs for $\mathcal{S}_n$. By Lemma~\ref{leSGDom03} and Fig.~\ref{SGTheta0}, Eq.~\eqref{SGset02} is established by exploiting the rotational symmetry of the Sierpi\'nski graph.

In a similar way, we can prove the remaining Eqs.~\eqref{SGset03}-\eqref{SGset05}.
\end{proof}

\section{RESULT COMPARISON AND ANALYSIS}

In the preceding two sections, we studied the edge domination number and   the number of MEDSs  for the pseudofracal scale-free web and the Sierpi\'nski graph,  both of which have the same number of vertices and edges. For both networks, we obtained  exact values for the edge domination number, as well as recursion solutions to the number of MEDSs.

Our results indicate that the edge domination number of the pseudofracal scale-free web is  one-ninth of all edges, which is three-fifths of the edge domination number of the  Sierpi\'nski gasket. Thus, the edge domination number for the pseudofracal scale-free web is much less than that associated with the Sierpi\'nski gasket. Actually, in addition to the edge domination number,  the number of MEDSs of  the pseudofracal scale-free web is also smaller than that corresponding to the Sierpi\'nski gasket. In Table~\ref{SetNo}, we list the  number of MEDSs of $\mathcal{G}_n$ and $\mathcal{S}_n$ for small $n$, which are obtained according to Theorems~\ref{NumEDSs} and~\ref{NumEDSg}.  From Table~\ref{SetNo}, we can see that for $n\geq 3$, the  number of MEDSs of $\mathcal{G}_n$ is always smaller than that of $\mathcal{S}_n$. However, for both graphs, the  number of MEDSs  grows exponentially with the number of edges $E_n$.

%%%%%%%%%%%%%%%%%%%%%%%%%%%%%%%%%%%%%
\begin{table*}
\tbl{The number of MEDSs in the pseudofractal scale-free web $\mathcal{G}_n$ and Sierpi\'nski gasket  $\mathcal{S}_n$ for small $n$.}
%\normalsize
%\centering
{\begin{tabular}{@{}|c|c|c|c|@{}}
\hline
\raisebox{-0.5ex}{$n$} & \raisebox{-0.5ex}{$E_n$} & \raisebox{-0.5ex}{MEDSs in $\mathcal{G}_n$} & \raisebox{-0.5ex}{MEDSs in $\mathcal{S}_n$}\\[0.5ex]
\hline
\hline
\raisebox{-0.5ex}{3} & \raisebox{-0.5ex}{$27$} & \raisebox{-0.5ex}{$1$} & \raisebox{-0.5ex}{$2$}\\[0.5ex]
\hline
\raisebox{-0.5ex}{4} & \raisebox{-0.5ex}{$81$} & \raisebox{-0.5ex}{$5$} & \raisebox{-0.5ex}{$737$}\\[0.5ex]
\hline
\raisebox{-0.5ex}{5} & \raisebox{-0.5ex}{$273$} & \raisebox{-0.5ex}{$223$} & \raisebox{-0.5ex}{$60406401428$}\\[0.5ex]
\hline
\raisebox{-0.5ex}{6} & \raisebox{-0.5ex}{$819$} & \raisebox{-0.5ex}{$12853595$} & \raisebox{-0.5ex}{$11968284390834034602027534554922752$}\\[0.5ex]
\hline
\end{tabular}}
\label{SetNo}
\end{table*}
%%%%%%%%%%%%%%%%%%%%%%%%%%%%%%

Because the size and the number of MEDSs of a graph are closely related to its structure, we argue that this distinction of MEDS problem between the pseudofractal scale-free web and the Sierpi\'nski graph highlights their structural disparity  and   can be heuristically  understood as follows.

Although both  networks have the same number of vertices and edges, the pseudofractal scale-free web is heterogeneous, the Sierpi\'nski graph homogeneous.  In the pseudofractal scale-free web, there exist some  high-degree vertices.  As shown above, for any MEDS of  the pseudofractal scale-free web, in order to minimize the size of the set, we should choose those edges incident to large-degree vertices as possible. However, once an edge incident to a large-degree vertex  is included in a MEDS, the other edges incident to it are not allowed to be in the set, which  substantially decreases the size of a  MEDS and the possible number of all MEDSs.  In the Sierpi\'nski graph,  all  vertices, except the three outmost ones, have degree of four. Thus, when constructing a MEDS, each plays a almost similar role. Any edge selected into a MEDS has less influence on  the edge domination number and the number of MEDSs.
Therefore, both the edge domination number and the number of MEDSs in the Sierpi\'nski graph are much larger than those corresponding to  the pseudofractal scale-free web.  Then, we naturally conclude that the heterogeneous  property has a great   effect  on the edge domination number and the number of MEDSs in a scale-free network.

We note that although we only consider  a particular  scale-free  network,  it is expected that the  edge domination set problem of other scale-free networks including real-world scale-free networks is qualitatively similar to that of the pseudofractal scale-free web. In other words, their edge domination number and the number of MEDSs are also much less,  compared with homogenous graphs.

\section{CONCLUSIONS}

We studied the edge domination number and  the number of MEDSs in the pseudofracal scale-free web and the Sierpi\'nski graph, which have the same number of vertices and edges. For both networks, by using their self-similarity we  determined explicit expressions for the edge domination number. For the former, the edge domination number  is  smaller taking up one-ninth  of  all edges in the network; while for the latter, the edge domination number  is larger taking up 5/27  of  all edges.  In addition, the number of MEDSs in the former network is also much less that for the latter network, but for both networks the number of MEDSs  grows exponentially with the total number of  edges.  Our work offers insight into applications of MEDSs in scale-free graphs.

Finally, it deserves to mention that the pseudofracal scale-free web is in fact constructed by iteratively using the triangulation operation on a complete  graph with three vertices~\cite{XiZhCo16b}.  Our computation method and process for computing the edge domination number and  the number of MEDSs are also applicable  to other graph operations, such as
subdivision~\cite{XiZhCo16}.

%\ack{This work was supported  in part by the National Key R \& D Program of China
%(No. 2018YFB1305104), the National Natural Science Foundation of China (No. 61872093),  Shanghai Municipal Science and Technology Major Project  (No.  2018SHZDZX01) and ZHANGJIANG LAB.}
%\section*{Acknowledgment}
\nonumsection{ACKNOWLEDGEMENTS} \noindent
This work was supported  in part by the National Natural Science Foundation of China (Nos. 61803248, 61872093, U20B2051, and U19A2066), the National Key R \& D Program of China
(No. 2018YFB1305104 and 2019YFB2101703), and Shanghai Municipal Science and Technology Major Project  (No.  2018SHZDZX01) and ZHANGJIANG LAB. Xiaotian Zhou was also supported by Fudan Undergraduate Research
Opportunities Program (FDUROP).

%and the Innovation Action Plan of Shanghai Science and Technology (Nos. 20222420800 and 20511102200).

%\textcolor[rgb]{0.00,0.00,1.00}{The authors are grateful to the anonymous reviewers for their valuable comments and suggestions, which have led to improvement of this paper.}

%\section*{References}

%\bibliographystyle{compj}% model1-num-names

\bibliographystyle{ws-ijait}
\bibliography{EdgeDomination}

\end{multicols}
\end{document}